\documentclass[10pt,twoside]{article}
\def\YEAR{\year}\newcount\VOL\VOL=\YEAR\advance\VOL by-1995
\def\firstpage{1}\def\lastpage{5}
\def\received{}\def\revised{}
\def\communicated{}

\makeatletter
\def\magnification{\afterassignment\m@g\count@}
\def\m@g{\mag=\count@\hsize6.5truein\vsize8.9truein\dimen\footins8truein}
\makeatother

\font\eightrm=cmr8
\font\caps=cmcsc10                    
\font\Caps=cmcsc10 scaled \magstep1   

\makeatletter
\@specialpagefalse\headheight=8.5pt
\def\DocMath{}
\renewcommand{\@evenhead}{%
    \ifnum\thepage>\lastpage\rlap{\thepage}\hfill%
    \else\rlap{\thepage}\slshape\leftmark\hfill{\caps\SAuthor}\hfill\fi}%
\renewcommand{\@oddhead}{%
    \ifnum\thepage=\firstpage{\DocMath\hfill\llap{\thepage}}%
    \else{\slshape\rightmark}\hfill{\caps\STitle}\hfill\llap{\thepage}\fi}%
\makeatother

\def\TSkip{\bigskip}
\newbox\TheTitle{\obeylines\gdef\GetTitle #1
\ShortTitle  #2
\SubTitle    #3
\Author      #4
\ShortAuthor #5
\EndTitle
{\setbox\TheTitle=\vbox{\baselineskip=20pt\let\par=\cr\obeylines%
\halign{\centerline{\Caps##}\cr\noalign{\medskip}\cr#1\cr}}%
	\copy\TheTitle\TSkip\TSkip%
\def\next{#2}\ifx\next\empty\gdef\STitle{#1}\else\gdef\STitle{#2}\fi%
\def\next{#3}\ifx\next\empty%
    \else\setbox\TheTitle=\vbox{\baselineskip=20pt\let\par=\cr\obeylines%
    \halign{\centerline{\caps##} #3\cr}}\copy\TheTitle\TSkip\TSkip\fi%
\centerline{\caps #4}\TSkip\TSkip%
\def\next{#5}\ifx\next\empty\gdef\SAuthor{#4}\else\gdef\SAuthor{#5}\fi%
\ifx\received\empty\relax
    \else\centerline{\eightrm Received: \received}\fi%
\ifx\revised\empty\TSkip%
    \else\centerline{\eightrm Revised: \revised}\TSkip\fi%
\ifx\communicated\empty\relax
    \else\centerline{\eightrm Communicated by \communicated}\fi\TSkip\TSkip%
\catcode'015=5}}\def\Title{\obeylines\GetTitle}
\def\Abstract{\begingroup\narrower
    \parskip=\medskipamount\parindent=0pt{\caps Abstract. }}
\def\EndAbstract{\par\endgroup\TSkip}

\long\def\MSC#1\EndMSC{\def\arg{#1}\ifx\arg\empty\relax\else
     {\par\narrower\noindent%
     2010 Mathematics Subject Classification: #1\par}\fi}

\long\def\KEY#1\EndKEY{\def\arg{#1}\ifx\arg\empty\relax\else
	{\par\narrower\noindent Keywords and Phrases: #1\par}\fi\TSkip}

\newbox\TheAdd\def\Addresses{\vfill\copy\TheAdd\vfill
    \ifodd\number\lastpage\vfill\eject\phantom{.}\vfill\eject\fi}
{\obeylines\gdef\GetAddress #1
\Address #2
\Address #3
\Address #4
\EndAddress
{\def\xs{4.3truecm}\parindent=0pt
\setbox0=\vtop{{\obeylines\hsize=\xs#1\par}}\def\next{#2}
\ifx\next\empty 
     \setbox\TheAdd=\hbox to\hsize{\hfill\copy0\hfill}
\else\setbox1=\vtop{{\obeylines\hsize=\xs#2\par}}\def\next{#3}
\ifx\next\empty 
     \setbox\TheAdd=\hbox to\hsize{\hfill\copy0\hfill\copy1\hfill}
\else\setbox2=\vtop{{\obeylines\hsize=\xs#3\par}}\def\next{#4}
\ifx\next\empty\ 
     \setbox\TheAdd=\vtop{\hbox to\hsize{\hfill\copy0\hfill\copy1\hfill}
                \vskip20pt\hbox to\hsize{\hfill\copy2\hfill}}
\else\setbox3=\vtop{{\obeylines\hsize=\xs#4\par}}
     \setbox\TheAdd=\vtop{\hbox to\hsize{\hfill\copy0\hfill\copy1\hfill}
	        \vskip20pt\hbox to\hsize{\hfill\copy2\hfill\copy3\hfill}}
\fi\fi\fi\catcode'015=5}}\gdef\Address{\obeylines\GetAddress}

\usepackage{amssymb}
\usepackage{mathrsfs}
\usepackage{amsfonts,amsthm}
\usepackage{amsmath,amscd}
\usepackage{amsxtra}     
\usepackage[all,cmtip]{xy}
\usepackage{epsfig}
\usepackage{verbatim}
\usepackage{color}
\usepackage{enumerate}
\usepackage{hyperref}
\usepackage{tikz}
\usetikzlibrary{arrows,matrix,shapes,trees}

\theoremstyle{plain}
 \newtheorem{thm}{Theorem}[section]
 \newtheorem{prop}{Proposition}[section]
 \newtheorem{lem}{Lemma}[section]
 \newtheorem{cor}{Corollary}[section]

 \newtheorem{lem'}{``Lemma''}[section]
 
 \newtheorem*{claim}{Claim}
\theoremstyle{definition}
 \newtheorem{ex}{Example}[section]
 \newtheorem{defn}{Definition}[section]
\theoremstyle{remark}
 \newtheorem{rmk}{Remark}[section]
 \numberwithin{equation}{section}

\newcommand{\N}{{\mathbb N}}
\newcommand{\Q}{{\mathbb Q}}
\newcommand{\Z}{{\mathbb Z}}

\newcommand{\Gm}{\mathbb{G}_m}

\newcommand{\Gml}{\mathbb{G}_{m,\mr{log}}}
\newcommand{\Gmlb}{\bar{\mathbb{G}}_{m,\mr{log}}}

\newcommand{\mr}{\mathrm}
\newcommand{\mc}{\mathcal}

\begin{document}
\Title
			Log abelian varieties over a log point
\ShortTitle
			Log abelian varieties over a log point
\SubTitle

\Author
		Heer Zhao
\ShortAuthor
		Heer Zhao  
\EndTitle

\Abstract
We study (weak) log abelian varieties with constant degeneration in the log flat topology. If the base is a log point, we further study the endomorphism algebras of log abelian varieties. In particular, we prove the dual short exact sequence for isogenies, Poincar\'e complete reducibility theorem for log abelian varieties, and the semi-simplicity of the endomorphism algebras of log abelian varieties.
\EndAbstract

\MSC   
Primary 14D06; Secondary 14K99, 11G99.
\EndMSC

\KEY
log abelian varieties with constant degeneration, endomorphism algebras, Poincar\'e complete reducibility theorem, dual short exact sequence for isogenies.
\EndKEY

\Address
    Heer Zhao
    Fakult\"at f\"ur Mathematik
    Universit\"at Duisburg-Essen
    Essen 45117 
    Germany
    	heer.zhao@gmail.com

\Address

\Address
\Address
\EndAddress
\section{Introduction}
\label{sec:introduction}

As stated in \cite{k-k-n2}, degenerating abelian varieties can not preserve group structure, properness, and smoothness at the same time. Log abelian variety is a construction aimed to make the impossible possible in the world of log geometry. The idea dates back to Kato's construction of log Tate curve in \cite[Sec. 2.2]{kat3}, in which he also conjectured the existence of a general theory of log abelian varieties. The theory finally comes true in \cite{k-k-n1} and \cite{k-k-n2}. 

Log abelian varieties are defined as certain sheaves in the classical \'etale topology in \cite{k-k-n2}, however the log flat topology is needed for studying some problems, for example finite group subobjects of log abelian varieties, $l$-adic realisations of log abelian varieties, logarithmic Dieudonn\'e theory of log abelian varieties and so on. In section \ref{sec2}, we prove that various classical \'etale sheaves from \cite{k-k-n2} are also sheaves for the log flat topology, in particular we prove that (weak) log abelian varieties with constant degeneration are sheaves for the log flat topology, see Theorem \ref{thm1.1}. We compute the first direct image sheaves of \'etale locally finite rank free constant sheaves, for changing to the log flat site from the classical \'etale site, in Lemma \ref{lem1.3}. This lemma can be considered as a supplement or generalisation of \cite[Thm. 4.1]{kat2}. We also reformulate some results from \cite[\S 2, \S 3 and \S 7]{k-k-n2} in the context of the log flat topology. 

In section \ref{sec3}, we focus on the case that the base is a log point. In this case, a log abelian variety is automatically a log abelian variety with constant degeneration. And only in this case, log abelian variety is the counterpart of abelian variety. While for general base, log abelian variety corresponds to abelian scheme. Now one may wonder if various results for abelian variety also hold for log abelian variety. We study isogenies and general homomorphisms between log abelian varieties over a log point. More precisely, we give several equivalent characterisations of isogeny in Proposition \ref{prop2.3}, and prove the dual short exact sequence in Theorem \ref{thm2.1}, Poincar\'e complete reducibility theorem for log abelian varieties in Theorem \ref{thm2.2}, and the finiteness of homomorphism group of log abelian varieties in Theorem \ref{thm2.4}, Corollary \ref{cor2.3}, Corollary \ref{cor2.4}, and Corollary \ref{cor2.5}.

\subsection*{Acknowledgement}
I am grateful to Professor Kazuya Kato for sending me a copy of the paper \cite{k-k-n4} which had been accepted but not yet published when the author started to work on this paper. I thank Professor Chikara Nakayama for very helpful communications concerning the paper \cite{k-k-n4}. I thank Professor Qing Liu for telling me the reference \cite[Rem. 5.4.7. (iii)]{bri1} about semi-abelian varieties. Part of this work was done during the author's informal stay at Professor Gebhard B\"ockle's Arbeitsgruppe, and I thank him for his hospitality and kindness. 

I would like to thank the anonymous referee, whose feedback has greatly improved this article. 

Part of this work has been supported by SFB/TR 45 ``Periods, moduli spaces
and arithmetic of algebraic varieties".

\section{Log abelian varieties with constant degeneration in the log flat topology}\label{sec2}
When dealing with finite subgroup schemes of abelian varieties, one needs to work with the flat topology. Similarly, the log flat topology is needed in the study of log finite group subobjects of log abelian varieties. However, log abelian varieties in \cite{k-k-n2} are defined in the classical \'etale topology. In this section, we are going to reformulate some results from \cite[\S 2, \S 3 and \S 7]{k-k-n2}, which are formulated in the context of classical \'etale topology, in the context of log flat topology. 

Throughout this section, let $S$ be any fs log scheme with its underlying scheme locally noetherian, and $(\mr{fs}/S)$ be the category of fs log schemes over $S$. The log schemes in this section will always be fs log schemes unless otherwise stated. Let $S^{\mr{cl}}_{\mr{\acute{E}t}}$ (resp. $S^{\mr{cl}}_{\mr{fl}}$, resp. $S^{\mr{log}}_{\mr{\acute{E}t}}$, resp. $S^{\mr{log}}_{\mr{fl}}$) be the classical \'etale site (resp. classical flat site, resp. log \'etale site, resp. log flat site)\footnote{Here we are following the terminology from \cite{kat2}. Note that in \cite{k-k-n4} $S^{\mr{cl}}_{\mr{\acute{E}t}}$ is called the strict \'etale site, while $S^{\mr{log}}_{\mr{\acute{E}t}}$ and $S^{\mr{log}}_{\mr{fl}}$ are called the Kummer log \'etale site and the Kummer log flat site respectively.} associated to the category $(\mr{fs}/S)$, and let $\delta=m\circ\varepsilon_{\mr{fl}}:S^{\mr{log}}_{\mr{fl}}\xrightarrow{\varepsilon_{\mr{fl}}} S^{\mr{cl}}_{\mr{fl}}\xrightarrow{m} S^{\mr{cl}}_{\mr{\acute{E}t}}$ be the canonical map of sites. For any inclusion $F\subset G$ of sheaves on $S^{\mr{log}}_{\mr{fl}}$, we denote by $G/F$ the quotient sheaf in the category of sheaves on $S^{\mr{log}}_{\mr{fl}}$ by convention, unless otherwise stated.

We start with the following lemma, which relates the Hom sheaves in the classical \'etale topology to the Hom sheaves in the log flat topology. Although this lemma is somehow trivial, we still formulate it due to its extensive use in this paper.

\begin{lem}\label{lem1.1}
Let $F,G$ be two sheaves on $S^{\mr{cl}}_{\mr{\acute{E}t}}$ which are also sheaves on $S^{\mr{log}}_{\mr{fl}}$. Then we have $\mc{H}om_{S^{\mr{cl}}_{\mr{\acute{E}t}}}(F,G)=\mc{H}om_{S^{\mr{log}}_{\mr{fl}}}(F,G)$, in particular $\mc{H}om_{S^{\mr{cl}}_{\mr{\acute{E}t}}}(F,G)$ is a sheaf on $S^{\mr{log}}_{\mr{fl}}$.
\end{lem}
\begin{proof}
This is clear.
\end{proof}

Now we recall some definitions from \cite{k-k-n2}. Let $G$ be a commutative group scheme over the underlying scheme of $S$ which is an extension of an abelian scheme $B$ by a torus $T$. Let $X$ be the character group of $T$ which is a locally constant sheaf of finite generated free $\Z$-modules for the classical \'etale topology. The sheaf $\Gml$ on $S^{\mr{cl}}_{\mr{\acute{E}t}}$ is defined by
$$\Gml(U)=\Gamma(U,M_U^{\mr{gp}}),$$
the sheaf $T_{\mr{log}}$ on $S^{\mr{cl}}_{\mr{\acute{E}t}}$ is defined by
$$T_{\mr{log}}:=\mc{H}om_{S^{\mr{cl}}_{\mr{\acute{E}t}}}(X,\Gml),$$
and the sheaf $G_{\mr{log}}$ is defined as the push-out of $T_{\mr{log}}\leftarrow T\rightarrow G$ in the category of sheaves on $S^{\mr{cl}}_{\mr{\acute{E}t}}$, see \cite[2.1]{k-k-n2}. 

\begin{prop}\label{prop1.1}
The sheaves $\Gml$, $X$, $T_{\mr{log}}$ and $G_{\mr{log}}$ on $S^{\mr{cl}}_{\mr{\acute{E}t}}$ are also sheaves for the log flat topology. Moreover, $T_{\mr{log}}$ can be alternatively defined as 
$$\mc{H}om_{S^{\mr{log}}_{\mr{fl}}}(X,\Gml),$$
and $G_{\mr{log}}$ can be alternatively defined as the push-out of $T_{\mr{log}}\leftarrow T\rightarrow G$ in the category of sheaves on $S^{\mr{log}}_{\mr{fl}}$.
\end{prop}
\begin{proof}
The statement for $\Gml$ is just \cite[Thm. 3.2]{kat2}, see also \cite[Cor. 2.22]{niz1}. Being representable by a group scheme, $X$ is a sheaf on $S^{\mr{log}}_{\mr{fl}}$ by \cite[Thm. 3.1]{kat2} and \cite[Thm. 5.2]{k-k-n4}. It follows then $T_{\mr{log}}=\mc{H}om_{S^{\mr{cl}}_{\mr{\acute{E}t}}}(X,\Gml)=\mc{H}om_{S^{\mr{log}}_{\mr{fl}}}(X,\Gml)$ is also a sheaf on $S^{\mr{log}}_{\mr{fl}}$. By its definition $G_{\mr{log}}$ fits into a short exact sequence $0\rightarrow T_{\mr{log}}\rightarrow G_{\mr{log}}\rightarrow B\rightarrow 0$ of sheaves on $S^{\mr{cl}}_{\mr{\acute{E}t}}$. Consider the following commutative diagram
$$
\xymatrix{
0\ar[r] &T_{\mr{log}}\ar[r]\ar[d]^{=} &G_{\mr{log}}\ar[r]\ar[d] &B\ar[r]\ar[d]^{=} &0 \\
0\ar[r] &T_{\mr{log}}\ar[r]  &\delta_*\delta^*G_{\mr{log}}\ar[r]  &B \ar[r] &R^1\delta_*T_{\mr{log}}
}$$
with exact rows in the category of sheaves on $S^{\mr{cl}}_{\mr{\acute{E}t}}$, where the vertical maps come from the adjunction $(\delta^*,\delta_*)$. The sheaf $R^1\delta_*T_{\mr{log}}$ is zero by Kato's logarithmic Hilbert 90, see \cite[Cor. 5.2]{kat2} or \cite[Thm. 3.20]{niz1}. It follows that the canonical map $G_{\mr{log}}\rightarrow \delta_*\delta^*G_{\mr{log}}$ is an isomorphism, whence $G_{\mr{log}}$ is a sheaf on $S^{\mr{log}}_{\mr{fl}}$. Since $G_{\mr{log}}$, as a push-out of $T_{\mr{log}}\leftarrow T\rightarrow G$ in the category of sheaves on $S^{\mr{cl}}_{\mr{\acute{E}t}}$, is already a sheaf on $S^{\mr{log}}_{\mr{fl}}$, it coincides with the push-out of $T_{\mr{log}}\leftarrow T\rightarrow G$ in the category of sheaves on $S^{\mr{log}}_{\mr{fl}}$.
\end{proof}

\begin{prop}\label{prop1.2}
We have canonical isomorphisms $$\mc{H}om_{S^{\mr{log}}_{\mr{fl}}}(X,\Gml/\Gm)\cong T_{\mr{log}}/T\cong G_{\mr{log}}/G.$$
\end{prop}
\begin{proof}
By Proposition \ref{prop1.1}, $G_{\mr{log}}$ is the push-out of $T_{\mr{log}}\leftarrow T\rightarrow G$ in the category of sheaves on $S^{\mr{log}}_{\mr{fl}}$, so we get a commutative diagram
\begin{equation*}
\xymatrix{
0\ar[r] &T\ar[r]\ar[d] &G\ar[r]\ar[d] &B\ar[r]\ar@{=}[d] &0 \\
0\ar[r] &T_{\mr{log}}\ar[r]  &G_{\mr{log}}\ar[r]  &B \ar[r] &0
}
\end{equation*}
with exact rows. Then the isomorphism $T_{\mr{log}}/T\cong G_{\mr{log}}/G$ follows. Applying the functor $\mc{H}om_{S^{\mr{log}}_{\mr{fl}}}(X,-)$ to the short exact sequence
$$0\rightarrow \Gm\rightarrow \Gml\rightarrow \Gml/\Gm\rightarrow 0,$$
we get a long exact sequence
$$0\rightarrow T\rightarrow T_{\mr{log}}\rightarrow \mc{H}om_{S^{\mr{log}}_{\mr{fl}}}(X,\Gml/\Gm)\rightarrow \mc{E}xt_{S^{\mr{log}}_{\mr{fl}}}(X,\Gm)$$
of sheaves on $S^{\mr{log}}_{\mr{fl}}$. Since $X$ is classical \'etale locally represented by a finite rank free abelian group, the sheaf $\mc{E}xt_{S^{\mr{log}}_{\mr{fl}}}(X,\Gm)$ is zero. It follows that the sheaf $\mc{H}om_{S^{\mr{log}}_{\mr{fl}}}(X,\Gml/\Gm)$ is canonically isomorphic to $T_{\mr{log}}/T$.
\end{proof}

It is obvious that the association of $G_{\mr{log}}$ to $G$ is functorial in $G$. Hence we have a natural map $\mr{Hom}_{S^{\mr{log}}_{\mr{fl}}}(G,G')\rightarrow \mr{Hom}_{S^{\mr{log}}_{\mr{fl}}}(G_{\mr{log}},G'_{\mr{log}})$, where $G'$ is another commutative group scheme which is an extension of an abelian scheme by a torus over the underlying scheme of $S$. The following proposition describes some properties of this map.

\begin{prop}\label{prop1.3}
\begin{enumerate}[(1)]
\item The association of $G_{\mr{log}}$ to $G$ is functorial in $G$.
\item The canonical map $\mr{Hom}_{S^{\mr{log}}_{\mr{fl}}}(G,G')\rightarrow \mr{Hom}_{S^{\mr{log}}_{\mr{fl}}}(G_{\mr{log}},G'_{\mr{log}})$ is an isomorphism.
\item For a group scheme $H$ of multiplicative type with character group $X_H$ over the underlying scheme of $S$, let $H_{\mr{log}}$ denote $\mc{H}om_{S^{\mr{log}}_{\mr{fl}}}(X_H,\Gml)$. Let $0\rightarrow H'\rightarrow H\rightarrow H''\rightarrow 0$ be a short exact sequence of group schemes of multiplicative type over the underlying scheme of $S$ such that their character groups are \'etale locally finite rank constant sheaves, then the sequences
$$0\rightarrow H'_{\mr{log}}\rightarrow H_{\mr{log}}\rightarrow H''_{\mr{log}}\rightarrow 0$$
and 
\begin{align*}
0&\rightarrow \mc{H}om_{S^{\mr{log}}_{\mr{fl}}}(X_{H'},\Gml/\Gm)\rightarrow \mc{H}om_{S^{\mr{log}}_{\mr{fl}}}(X_H,\Gml/\Gm) \\
&\rightarrow \mc{H}om_{S^{\mr{log}}_{\mr{fl}}}(X_{H''},\Gml/\Gm)\rightarrow 0
\end{align*}
are both exact.
\item If $G\rightarrow G'$ is injective, so is $G_{\mr{log}}\rightarrow G'_{\mr{log}}$.
\item If $G\rightarrow G'$ is surjective, so is $G_{\mr{log}}\rightarrow G'_{\mr{log}}$.
\item Let $0\rightarrow G'\rightarrow G\rightarrow G''\rightarrow 0$ be a short exact sequence of semi-abelian schemes over the underlying scheme of $S$, such that $G'$ (resp. $G$, resp. $G''$) is an extension of an abelian scheme $B'$ (resp. $B$, resp. $B''$) by a torus $T'$ (resp. $T$, resp. $T''$). Then we have a short exact sequence 
$0\rightarrow G'_{\mr{log}}\rightarrow G_{\mr{log}}\rightarrow G''_{\mr{log}}\rightarrow 0$.
\end{enumerate}
\end{prop}
\begin{proof}
Part (1) is clear. The isomorphism of part (2) follows from \cite[Prop. 2.5]{k-k-n2}.

We prove part (3). Since we have a long exact sequence $$0\rightarrow H'_{\mr{log}}\rightarrow H_{\mr{log}}\rightarrow H''_{\mr{log}}\rightarrow \mc{E}xt_{S^{\mr{log}}_{\mr{fl}}}(X_{H'},\Gml),$$ it suffices to show $\mc{E}xt_{S^{\mr{log}}_{\mr{fl}}}(X_{H'},\Gml)=0$. Since $\mc{E}xt_{S^{\mr{log}}_{\mr{fl}}}(\Z,\Gml)=0$, we are further reduced to show $\mc{E}xt_{S^{\mr{log}}_{\mr{fl}}}(\Z/n\Z,\Gml)=0$ for any positive integer $n$. The short exact sequence $0\rightarrow \Z\xrightarrow{n} \Z\rightarrow \Z/n\Z\rightarrow 0$ gives rise to a long exact sequence $0\rightarrow \mc{H}om_{S^{\mr{log}}_{\mr{fl}}}(\Z/n\Z,\Gml)\rightarrow \Gml\xrightarrow{n} \Gml\rightarrow \mc{E}xt_{S^{\mr{log}}_{\mr{fl}}}(\Z/n\Z,\Gml)\rightarrow 0$.
Since $\Gml\xrightarrow{n} \Gml$ is surjective, the sheaf $\mc{E}xt_{S^{\mr{log}}_{\mr{fl}}}(\Z/n\Z,\Gml)$ must be zero. The other short exact sequence is proved similarly.

We prove part (4). Since $G\rightarrow G'$ is injective, then the corresponding map $T\rightarrow T'$ on the torus parts is also injective and the corresponding map $X'\rightarrow X$ on the character groups is surjective. It follows that the induced map $$G_{\mr{log}}/G=\mc{H}om_{S^{\mr{log}}_{\mr{fl}}}(X,\Gml/\Gm)\rightarrow \mc{H}om_{S^{\mr{log}}_{\mr{fl}}}(X',\Gml/\Gm)=G'_{\mr{log}}/G'$$ is injective. Hence $G_{\mr{log}}\rightarrow G'_{\mr{log}}$ is injective.

Now we prove part (5). Let $f$ denote the map $G\rightarrow G'$. Consider the torus and abelian variety decomposition of $f$
$$\xymatrix{
0\ar[r] &T\ar[r]\ar[d]^{f_{\mr{t}}} &G\ar[r]\ar[d]^f &B\ar[r]\ar[d]^{f_{\mr{ab}}} &0  \\
0\ar[r] &T'\ar[r] &G'\ar[r] &B'\ar[r] &0 .
}$$
We first show that ${f_{\mr{t}}}$ is surjective. Assume that the underlying scheme of $S$ is a point. The snake lemma gives an exact sequence $\mr{Ker}({f_{\mr{ab}}})\rightarrow\mr{Coker}({f_{\mr{t}}})\rightarrow 0$. Since $\mr{Coker}({f_{\mr{t}}})$ is a torus and the reduced neutral component of $\mr{Ker}({f_{\mr{ab}}})$ is an abelian variety by \cite[Lem. 3.3.7]{bri1}, we must have $\mr{Coker}({f_{\mr{t}}})=0$. Hence $f_{\mr{t}}$ is surjective. In the general case, $f_{\mr{t}}$ is fiberwise surjective, hence it is also set-theoretically surjective. The fibers of $f_{\mr{t}}$ over $S$ are all flat, hence $f_{\mr{t}}$ is flat by the fiberwise criterion of flatness, see \cite[Cor. 11.3.11]{egaIV-3}. Then $f_{\mr{t}}$ is faithfully flat, hence it is surjective. Then we get a short exact sequence $0\rightarrow X'\rightarrow X\rightarrow X/X'\rightarrow 0$ of \'etale locally constant sheaves. Applying the functor $\mc{H}om_{S^{\mr{log}}_{\mr{fl}}}(-,\Gml/\Gm)$ to this short exact sequence, we get a long exact sequence 
$$\rightarrow G_{\mr{log}}/G\rightarrow G'_{\mr{log}}/G'\rightarrow \mc{E}xt_{S^{\mr{log}}_{\mr{fl}}}(X/X',\Gml/\Gm).$$
Let $Z_{\mr{tor}}$ be the torsion part of $X/X'$, and let $n$ be a positive integer such that $nZ_{\mr{tor}}=0$. Since the multiplication-by-$n$ map on $\Gml/\Gm$ is an isomorphism, we get that the sheaf $\mc{E}xt_{S^{\mr{log}}_{\mr{fl}}}(Z_{\mr{tor}},\Gml/\Gm)$ is zero. The torsion-free nature of $(X/X')/Z_{\mr{tor}}$ implies $\mc{E}xt_{S^{\mr{log}}_{\mr{fl}}}((X/X')/Z_{\mr{tor}},\Gml/\Gm)=0$, hence $\mc{E}xt_{S^{\mr{log}}_{\mr{fl}}}(X/X',\Gml/\Gm)=0$. It follows that $G_{\mr{log}}/G\rightarrow G'_{\mr{log}}/G'$ is surjective, hence $G_{\mr{log}}\rightarrow G'_{\mr{log}}$ is surjective.

At last, we prove part (6). Consider the following commutative diagram
$$\xymatrix{
&0\ar[d] &0\ar[d] &0\ar[d]  \\
0\ar[r] &G'\ar[r]\ar[d] &G'_{\mr{log}}\ar[r]\ar[d] &\mc{H}om_{S^{\mr{log}}_{\mr{fl}}}(X',\Gmlb)\ar[r]\ar[d] &0  \\
0\ar[r] &G\ar[r]\ar[d]  &G_{\mr{log}}\ar[r]\ar[d]  &\mc{H}om_{S^{\mr{log}}_{\mr{fl}}}(X,\Gmlb) \ar[r]\ar[d] &0   \\
0\ar[r] &G''\ar[r]\ar[d]  &G''_{\mr{log}}\ar[r]\ar[d]  &\mc{H}om_{S^{\mr{log}}_{\mr{fl}}}(X'',\Gmlb) \ar[r]\ar[d] &0  \\
&0 &0 &0
}$$
with the first column and all rows exact, where $\Gmlb$ denotes $\Gml/\Gm$. The maps $G'\rightarrow G\rightarrow G''$ induce $T'\rightarrow T\rightarrow T''$, furthermore $X'\leftarrow X\leftarrow X''$, lastly the third column of the diagram. Although $0\rightarrow X''\rightarrow X\rightarrow X'\rightarrow 0$ is not necessarily exact, it gives two exact sequences $0\rightarrow Z\rightarrow X\rightarrow X'\rightarrow 0$ and $0\rightarrow X''\rightarrow Z\rightarrow Z/X''\rightarrow 0$, where $Z:=\mr{Ker}(X\rightarrow X')$ is \'etale locally a finite rank free constant sheaf and $Z/X''$ is \'etale locally a finite torsion constant sheaf. By part (3), we get two short exact sequences
$$0\rightarrow \mc{H}om_{S^{\mr{log}}_{\mr{fl}}}(X',\Gmlb)\rightarrow \mc{H}om_{S^{\mr{log}}_{\mr{fl}}}(X,\Gmlb)\rightarrow \mc{H}om_{S^{\mr{log}}_{\mr{fl}}}(Z,\Gmlb)\rightarrow 0$$
and
\begin{align*}
0\rightarrow &\mc{H}om_{S^{\mr{log}}_{\mr{fl}}}(Z/X'',\Gmlb)\rightarrow \mc{H}om_{S^{\mr{log}}_{\mr{fl}}}(Z,\Gmlb)  \\
\rightarrow &\mc{H}om_{S^{\mr{log}}_{\mr{fl}}}(X'',\Gmlb)\rightarrow 0.
\end{align*}
But $\mc{H}om_{S^{\mr{log}}_{\mr{fl}}}(Z/X'',\Gmlb)=0$, it follows that the third column of the diagram is exact. So is the middle column.
\end{proof}

Recall that in \cite[Def. 2.2]{k-k-n2}, a log 1-motive $M$ over $S^{\mr{cl}}_{\mr{\acute{E}t}}$ is defined as a two-term complex $[Y\xrightarrow{u}G_{\mr{log}}]$ in the category of sheaves on $S^{\mr{cl}}_{\mr{\acute{E}t}}$, with the degree $-1$ term $Y$ an \'etale locally constant sheaf of finitely generated free abelian groups and the degree 0 term $G_{\mr{log}}$ as above. Since both $Y$ and $G_{\mr{log}}$ are sheaves on $S^{\mr{log}}_{\mr{fl}}$, $M$ can also be defined as a two-term complex $[Y\xrightarrow{u}G_{\mr{log}}]$ in the category of sheaves on $S^{\mr{log}}_{\mr{fl}}$. Parallel to \cite[2.3]{k-k-n2}, we have a natural pairing
\begin{equation}\label{eq1.1}
<,>:X\times Y\rightarrow X\times (G_{\mr{log}}/G)=X\times \mc{H}om_{S^{\mr{log}}_{\mr{fl}}}(X,\Gml/\Gm)\rightarrow \Gml/\Gm.
\end{equation}
It is clear that our pairing is induced from the one of \cite[2.3]{k-k-n2}. 

By our convention, $T_{\mr{log}}/T$ denotes the quotient in the category of sheaves on $S^{\mr{log}}_{\mr{fl}}$. For the quotient of $T\subset T_{\mr{log}}$ in the category of sheaves on $S^{\mr{cl}}_{\mr{\acute{E}t}}$, we use the notation $(T_{\mr{log}}/T)_{S^{\mr{cl}}_{\mr{\acute{E}t}}}$. Now we assume that the pairing (\ref{eq1.1}) is admissible (see \cite[7.1]{k-k-n2} for the definition of admissibility), in other words the log 1-motive $M$ is admissible. Recall that in \cite[3.1]{k-k-n2}, the subgroup sheaf $\mc{H}om_{S^{\mr{cl}}_{\mr{\acute{E}t}}}(X,(\Gml/\Gm)_{S^{\mr{cl}}_{\mr{\acute{E}t}}})^{(Y)}$ of the sheaf $\mc{H}om_{S^{\mr{cl}}_{\mr{\acute{E}t}}}(X,(\Gml/\Gm)_{S^{\mr{cl}}_{\mr{\acute{E}t}}})$ on $S^{\mr{cl}}_{\mr{\acute{E}t}}$ is defined by 
$$\begin{aligned}
& \mc{H}om_{S^{\mr{cl}}_{\mr{\acute{E}t}}}(X,(\Gml/\Gm)_{S^{\mr{cl}}_{\mr{\acute{E}t}}})^{(Y)}(U):= \\
& \{\varphi \in\mc{H}om_{S^{\mr{cl}}_{\mr{\acute{E}t}}}(X,(\Gml/\Gm)_{S^{\mr{cl}}_{\mr{\acute{E}t}}})(U)\,|\,\text{for every $u\in U$ and $x\in X_{\bar{u}}$},  \\
& \text{there exist $y_{u,x},y'_{u,x}\in Y_{\bar{u}}$ such that $<x,y_{u,x}>|\varphi_{\bar{u}}(x)|<x,y'_{u,x}>$}\}.
\end{aligned}$$
Here, $\bar{u}$ denotes a classical \'etale geometric point above $u$, and for $a,b\in (M_U^{\mr{gp}}/\mc{O}_U^{\times})_{\bar{u}}$, $a|b$ means $a^{-1}b\in (M_U/\mc{O}_U^{\times})_{\bar{u}}$.

It is natural to define the analogue of $\mc{H}om_{S^{\mr{cl}}_{\mr{\acute{E}t}}}(X,(\Gml/\Gm)_{S^{\mr{cl}}_{\mr{\acute{E}t}}})^{(Y)}$ in the log flat topology. We need the following lemma first.

\begin{lem}\label{lem1.2}
Let $\delta:S^{\mr{log}}_{\mr{fl}}\rightarrow S^{\mr{cl}}_{\mr{\acute{E}t}}$ be the canonical map between these two sites.
\begin{enumerate}[(1)]
\item $\delta_*(\Gml/\Gm)=(\Gml/\Gm)_{S^{\mr{cl}}_{\mr{\acute{E}t}}}\otimes_{\Z}\Q$.
\item Let $H$ be a commutative group scheme over the underlying scheme of $S$ with connected fibres. Then $\mr{Hom}_{S^{\mr{log}}_{\mr{fl}}}(H,\Gml/\Gm)=0$.
\end{enumerate}
\end{lem}
\begin{proof}
We denote the sheaf $\Gml/\Gm$ on $S^{\mr{log}}_{\mr{fl}}$ by $\Gmlb$. For any positive integer $n$, we have the following commutative diagram
$$
\xymatrix{
&&0\ar[d] &0\ar[d]  \\
0\ar[r] &\Z/n(1)\ar[r]\ar@{=}[d] &\Gm\ar[r]^{n}\ar[d] &\Gm\ar[r]\ar[d] &0  \\
0\ar[r] &\Z/n(1)\ar[r] &\Gml\ar[r]^{n}\ar[d] &\Gml\ar[r]\ar[d] &0  \\
&&\Gmlb \ar[r]^{n}_{\cong}\ar[d] &\Gmlb\ar[d]  \\
&&0 &0
}
$$
with exact rows and columns, where $\Z/n(1)$ denotes the group scheme of $n$-th roots of unity. Applying the functor $\varepsilon_{\mr{fl}*}$ to the above diagram, we get a new commutative diagram
$$
\xymatrix{
&0\ar[r] &\Z/n(1)\ar[r]\ar@{=}[d] &\Gm\ar[r]^n\ar@{^{(}->}[d] &\Gm\ar@{^{(}->}[d]   \\
&0\ar[r] &\Z/n(1)\ar[r] &\Gml\ar[r]^n\ar[d]^{\alpha} &\Gml\ar[d]^{\gamma}   \\
&&&\varepsilon_{\mr{fl}*}\Gmlb\ar[r]^n_{\cong}\ar[d]^{\beta} &\varepsilon_{\mr{fl}*}\Gmlb\ar[d]^{\delta}  \\
\Gm\ar[r]^n\ar@{^{(}->}[d] &\Gm\ar[r]\ar@{^{(}->}[d] &R^1\varepsilon_{\mr{fl}*}\Z/n(1)\ar[r]^{\eta}\ar@{=}[d] &R^1\varepsilon_{\mr{fl}*}\Gm\ar[r]^n\ar[d] &R^1\varepsilon_{\mr{fl}*}\Gm  \\
\Gml\ar[r]^n &\Gml\ar[r]^-{\theta} &R^1\varepsilon_{\mr{fl}*}\Z/n(1)\ar[r]&R^1\varepsilon_{\mr{fl}*}\Gml
}
$$
with exact rows and columns. Since the map $\Gm\xrightarrow{n}\Gm$ is surjective and $R^1\varepsilon_{\mr{fl}*}\Gml=0$, we get a new commutative diagram
$$\xymatrix{
&&&\mc{G}\otimes_{\Z}\Z/n\ar[d]^{\bar{\theta}}_{\cong}   \\
0\ar[r] &\mc{G}\ar[r]^n\ar@{=}[d] &\mc{G}\ar@{^{(}->}[d]^{\frac{1}{n}\gamma}\ar[r]^-{\xi}\ar[ru]^{\omega} &R^1\varepsilon_{\mr{fl}*}\Z/n(1)\ar[r]\ar@{^{(}->}[d]^{\eta} &0  \\
0\ar[r] &\mc{G} \ar[r]^-{\bar{\alpha}} &\varepsilon_{\mr{fl}*}\Gmlb\ar[r]^{\beta} &R^1\varepsilon_{\mr{fl}*}\Gm\ar[r] &0
}$$
with exact rows, where $\mc{G}$ denotes $(\Gml/\Gm)_{S^{\mr{cl}}_{\mr{fl}}}$, $\bar{\alpha}$ (resp. $\bar{\theta}$) is the canonical map induced by $\alpha$ (resp. $\theta$), $\omega$ is the canonical projection map and $\xi$ is the unique map guaranteed by $n\beta\circ (\frac{1}{n}\gamma)=\delta\circ (n(\frac{1}{n}\gamma))=\delta\circ\gamma=0$. Taking colimit of the above diagram with respect to $n$, we get a commutative diagram
$$\xymatrix{
0\ar[r] &\mc{G}\ar[r]\ar@{=}[d] &\mc{G}\otimes_{\Z}\Q\ar[d]\ar[r] &\mc{G}\otimes_{\Z}\Q/\Z\ar[r]\ar[d] &0  \\
0\ar[r] &\mc{G} \ar[r]^-{\bar{\alpha}} &\varepsilon_{\mr{fl}*}\Gmlb\ar[r]^{\beta} &R^1\varepsilon_{\mr{fl}*}\Gm\ar[r] &0
}$$
with exact rows. Since the map $\mc{G}\otimes_{\Z}\Q/\Z\rightarrow R^1\varepsilon_{\mr{fl}*}\Gm$ is an isomorphism by Kato's theorem \cite[Thm. 4.1]{kat2} (see also \cite[Thm. 3.12]{niz1}),
we get $\mc{G}\otimes_{\Z}\Q\cong\varepsilon_{\mr{fl}*}(\Gml/\Gm)$. Then 
\begin{align*}
\delta_*(\Gml/\Gm)=m_*\varepsilon_{\mr{fl}*}(\Gml/\Gm)&=m_*(\mc{G}\otimes_{\Z}\Q)    \\
&=(\Gml/\Gm)_{S^{\mr{cl}}_{\mr{\acute{E}t}}}\otimes_{\Z}\Q,
\end{align*}
where the last equality follows from the following fact: for any $U\in (\mr{fs}/S)$, the sheaf $M_U^{\mr{gp}}/\mc{O}_U^{\times}$ on the small \'etale site of $U$ is constructible. This proves part (1).

Now we prove part (2) which corresponds to \cite[Lem. 6.1.1]{k-k-n2}. We have 
\begin{align*}
\mr{Hom}_{S^{\mr{log}}_{\mr{fl}}}(H,\Gmlb)&=\mr{Hom}_{S^{\mr{cl}}_{\mr{\acute{E}t}}}(H,\delta_*\Gmlb)  \\
&=\mr{Hom}_{S^{\mr{cl}}_{\mr{\acute{E}t}}}(H,(\Gml/\Gm)_{S^{\mr{cl}}_{\mr{\acute{E}t}}}\otimes_{\Z}\Q).
\end{align*}
By the same argument of the proof of \cite[Lem. 6.1.1]{k-k-n2}, we have
$$\mr{Hom}_{S^{\mr{cl}}_{\mr{\acute{E}t}}}(H,(\Gml/\Gm)_{S^{\mr{cl}}_{\mr{\acute{E}t}}}\otimes_{\Z}\Q)=0.$$
Hence part (2) is proved.
\end{proof}

Now we define the analogue of $\mc{H}om_{S^{\mr{log}}_{\mr{\acute{E}t}}}(X,(\Gml/\Gm)_{S^{\mr{cl}}_{\mr{\acute{E}t}}})^{(Y)}$. It is the subgroup sheaf $\mc{H}om_{S^{\mr{log}}_{\mr{fl}}}(X,\Gml/\Gm)^{(Y)}$ of the sheaf $\mc{H}om_{S^{\mr{log}}_{\mr{fl}}}(X,\Gml/\Gm)$ on $S^{\mr{log}}_{\mr{fl}}$ given by
$$\begin{aligned}
& \mc{H}om_{S^{\mr{log}}_{\mr{fl}}}(X,\Gml/\Gm)^{(Y)}(U):= \\
& \{\varphi \in\mc{H}om_{S^{\mr{log}}_{\mr{fl}}}(X,\Gml/\Gm)(U)\,|\,\text{after pushing forward to $U^{\mr{cl}}_{\mr{\acute{E}t}}$,}  \\ 
& \text{for every $u\in U$ and $x\in X_{\bar{u}}$, there exist $y_{u,x},y'_{u,x}\in Y_{\bar{u}}$ such that}  \\
& <x,y_{u,x}>|\varphi_{\bar{u}}(x)|<x,y'_{u,x}>\}.
\end{aligned}$$
Here $\bar{u}$ still denotes a classical \'etale geometric point above $u$. Let $F:=\delta_{*}(\Gml/\Gm)=(\Gml/\Gm)_{S^{\mr{cl}}_{\mr{\acute{E}t}}}\otimes_{\Z}\Q$ with $\delta$ the canonical map $U^{\mr{log}}_{\mr{fl}}\rightarrow U^{\mr{cl}}_{\mr{\acute{E}t}}$. For $a,b\in (M_U^{\mr{gp}}/\mc{O}_U^{\times})_{\bar{u}}\otimes_{\Z}\Q$, $a|b$ means $a^{-1}b=\alpha\otimes_{\Z}r$ for some $\alpha\in (M_U/\mc{O}_U^{\times})_{\bar{u}}$ and $r\in\Q$.

\begin{rmk}\label{rmk1.1}
In \cite[7.1]{k-k-n2}, admissibility and non-degeneracy are defined for pairings into $(\Gml/\Gm)_{S^{\mr{cl}}_{\mr{\acute{E}t}}}$ in the classical \'etale site on $(\mr{fs}/S)$. We can define admissibility and non-degeneracy for pairings into $\Gml/\Gm$ on the log flat site in the same way. Since both $X$ and $Y$ are classical \'etale locally constant sheaves of finite rank free abelian groups, the definitions of admissibility and non-degeneracy are independent of the choice of the topology.
\end{rmk}

The next lemma compares the sheaf $\mc{H}om_{S^{\mr{cl}}_{\mr{\acute{E}t}}}(X,(\Gml/\Gm)_{S^{\mr{cl}}_{\mr{\acute{E}t}}})^{(Y)}$ on $S^{\mr{cl}}_{\mr{\acute{E}t}}$ with the sheaf $\mc{H}om_{S^{\mr{log}}_{\mr{fl}}}(X,\Gml/\Gm)^{(Y)}$ on $S^{\mr{log}}_{\mr{fl}}$.

\begin{lem}\label{lem1.2half}
Let $X,Y$ be two free abelian groups of finite rank, $<,>:X\times Y\rightarrow (\Gml/\Gm)_{S^{\mr{cl}}_{\mr{\acute{E}t}}}$ an admissible pairing on $S^{\mr{cl}}_{\mr{\acute{E}t}}$. Let
$$\mc{Q}_{\mr{cl}}:=\mc{H}om_{S^{\mr{cl}}_{\mr{\acute{E}t}}}(X,(\Gml/\Gm)_{S^{\mr{cl}}_{\mr{\acute{E}t}}})^{(Y)},\,\mc{Q}:=\mc{H}om_{S^{\mr{log}}_{\mr{fl}}}(X,\Gml/\Gm)^{(Y)},$$
and $\delta:S^{\mr{log}}_{\mr{fl}}\rightarrow S^{\mr{cl}}_{\mr{\acute{E}t}}$ the canonical map between these two sites. Then we have $\mc{Q}=\delta^*\mc{Q}_{\mr{cl}}$ and $\delta_*\mc{Q}=\mc{Q}_{\mr{cl}}\otimes_{\Z}\Q$.
\end{lem}
\begin{proof}
Denote $\Gml/\Gm$ by $\Gmlb$. We have  
$$\delta^*(\Gml/\Gm)_{S^{\mr{cl}}_{\mr{\acute{E}t}}}=\Gmlb,$$
and $$\delta_*(\Gmlb)=(\Gml/\Gm)_{S^{\mr{cl}}_{\mr{\acute{E}t}}}\otimes_{\Z}\Q$$
by part (1) of Lemma \ref{lem1.2}, hence $$\delta^*\mc{H}om_{S^{\mr{cl}}_{\mr{\acute{E}t}}}(X,(\Gml/\Gm)_{S^{\mr{cl}}_{\mr{\acute{E}t}}})=\mc{H}om_{S^{\mr{log}}_{\mr{fl}}}(X,\Gmlb),$$
and $$\delta_*\mc{H}om_{S^{\mr{log}}_{\mr{fl}}}(X,\Gmlb)=\mc{H}om_{S^{\mr{cl}}_{\mr{\acute{E}t}}}(X,(\Gml/\Gm)_{S^{\mr{cl}}_{\mr{\acute{E}t}}})\otimes_{\Z}\Q.$$ 
Then by the definition of $\mc{Q}$ and $\mc{Q}_{\mr{cl}}$, we get $\mc{Q}=\delta^*\mc{Q}_{\mr{cl}}$ and $\delta_*\mc{Q}=\mc{Q}_{\mr{cl}}\otimes_{\Z}\Q$.
\end{proof}

Recall that in \cite[3.2, Thm. 7.3]{k-k-n2}, $G_{\mr{log}}^{(Y)}\subset G_{\mr{log}}$ (resp. $T_{\mr{log}}^{(Y)}\subset T_{\mr{log}}$) on $S^{\mr{cl}}_{\mr{\acute{E}t}}$ is defined to be the inverse image of $\mc{H}om_{S^{\mr{cl}}_{\mr{\acute{E}t}}}(X,(\Gml/\Gm)_{S^{\mr{cl}}_{\mr{\acute{E}t}}})^{(Y)}$ under the map 
$$G_{\mr{log}}\rightarrow (G_{\mr{log}}/G)_{S^{\mr{cl}}_{\mr{\acute{E}t}}}\cong\mc{H}om_{S^{\mr{cl}}_{\mr{\acute{E}t}}}(X,(\Gml/\Gm)_{S^{\mr{cl}}_{\mr{\acute{E}t}}})$$
$$\textrm{(resp. } T_{\mr{log}}\rightarrow (T_{\mr{log}}/T)_{S^{\mr{cl}}_{\mr{\acute{E}t}}}\cong\mc{H}om_{S^{\mr{cl}}_{\mr{\acute{E}t}}}(X,(\Gml/\Gm)_{S^{\mr{cl}}_{\mr{\acute{E}t}}})\textrm{).}$$
We could also consider the inverse image sheaf of $\mc{H}om_{S^{\mr{log}}_{\mr{fl}}}(X,\Gml/\Gm)^{(Y)}$ under the map 
$$G_{\mr{log}}\rightarrow G_{\mr{log}}/G\cong\mc{H}om_{S^{\mr{log}}_{\mr{fl}}}(X,\Gml/\Gm)$$
$$\textrm{(resp. }T_{\mr{log}}\rightarrow T_{\mr{log}}/T\cong\mc{H}om_{S^{\mr{log}}_{\mr{fl}}}(X,\Gml/\Gm)\textrm{).}$$
The following proposition states that these two constructions coincide.

\begin{prop}\label{prop1.4}
\begin{enumerate}[(1)]
\item The sheaf $G_{\mr{log}}^{(Y)}$ on $S^{\mr{cl}}_{\mr{\acute{E}t}}$ is also a sheaf on $S^{\mr{log}}_{\mr{fl}}$.
\item The sheaf $G_{\mr{log}}^{(Y)}$ fits into a canonical short exact sequence 
\begin{equation}\label{eq1.2}
0\rightarrow G\rightarrow G_{\mr{log}}^{(Y)}\rightarrow \mc{H}om_{S^{\mr{log}}_{\mr{fl}}}(X,\Gml/\Gm)^{(Y)}\rightarrow 0
\end{equation}
of sheaves on $S^{\mr{log}}_{\mr{fl}}$. 
\item The association of $G_{\mr{log}}^{(Y)}$ to a log 1-motive $M=[Y\rightarrow G_{\mr{log}}]$ is functorial.
\end{enumerate}
\end{prop}
\begin{proof}
Let $T$ (resp. $B$) be the torus (resp. abelian scheme) part of $G$, then we have a short exact sequence $0\rightarrow T_{\mr{log}}^{(Y)}\rightarrow G_{\mr{log}}^{(Y)}\rightarrow B\rightarrow 0$ of sheaves on $S^{\mr{cl}}_{\mr{\acute{E}t}}$. To show that $G_{\mr{log}}^{(Y)}$ on $S^{\mr{cl}}_{\mr{\acute{E}t}}$ is a sheaf on $S^{\mr{log}}_{\mr{fl}}$, it suffices to show that $T_{\mr{log}}^{(Y)}$ is so. By \cite[7.7]{k-k-n2}, locally on $S^{\mr{cl}}_{\mr{\acute{E}t}}$ the sheaf $T_{\mr{log}}^{(Y)}$ is a union of representable sheaves. Hence it is also a sheaf on $S^{\mr{log}}_{\mr{fl}}$. So part (1) is proven.


By the definition of $G_{\mr{log}}^{(Y)}$, we have a pullback diagram 
$$\xymatrix{
0\ar[r] &G\ar[r]\ar@{=}[d] &G_{\mr{log}}^{(Y)}\ar[r]\ar@{^(->}[d] &\mc{H}om_{S^{\mr{cl}}_{\mr{\acute{E}t}}}(X,(\Gml/\Gm)_{S^{\mr{cl}}_{\mr{\acute{E}t}}})^{(Y)}\ar[r]\ar@{^(->}[d] &0 \\
0\ar[r] &G\ar[r]  &G_{\mr{log}}\ar[r]  &\mc{H}om_{S^{\mr{cl}}_{\mr{\acute{E}t}}}(X,(\Gml/\Gm)_{S^{\mr{cl}}_{\mr{\acute{E}t}}}) \ar[r] &0
}$$
in the category of sheaves on $S^{\mr{cl}}_{\mr{\acute{E}t}}$. Since $G$, $G_{\mr{log}}^{(Y)}$ and $G_{\mr{log}}$ are all sheaves on $S^{\mr{log}}_{\mr{fl}}$, applying the functor $\delta^*$ to the above commutative diagram, we get the following commutative diagram
$$\xymatrix{
0\ar[r] &G\ar[r]\ar@{=}[d] &G_{\mr{log}}^{(Y)}\ar[r]\ar@{^(->}[d] &\delta^*\mc{H}om_{S^{\mr{cl}}_{\mr{\acute{E}t}}}(X,(\Gml/\Gm)_{S^{\mr{cl}}_{\mr{\acute{E}t}}})^{(Y)}\ar[r]\ar@{^(->}[d] &0 \\
0\ar[r] &G\ar[r]  &G_{\mr{log}}\ar[r]  &\delta^*\mc{H}om_{S^{\mr{cl}}_{\mr{\acute{E}t}}}(X,(\Gml/\Gm)_{S^{\mr{cl}}_{\mr{\acute{E}t}}}) \ar[r] &0.
}$$
Since we have canonical isomorphisms
$$\delta^*\mc{H}om_{S^{\mr{cl}}_{\mr{\acute{E}t}}}(X,(\Gml/\Gm)_{S^{\mr{cl}}_{\mr{\acute{E}t}}})\cong \mc{H}om_{S^{\mr{log}}_{\mr{fl}}}(X,\Gml/\Gm)$$
and
$$\delta^*\mc{H}om_{S^{\mr{cl}}_{\mr{\acute{E}t}}}(X,(\Gml/\Gm)_{S^{\mr{cl}}_{\mr{\acute{E}t}}})^{(Y)}\cong \mc{H}om_{S^{\mr{log}}_{\mr{fl}}}(X,\Gml/\Gm)^{(Y)},$$
part (2) follows. 

Now we prove part (3). It is enough to prove that for a given  homomorphism $(f_{-1},f_0):M=[Y\rightarrow G_{\mr{log}}]\rightarrow M'=[Y'\rightarrow G'_{\mr{log}}]$, the composition $G_{\mr{log}}^{(Y)}\hookrightarrow G_{\mr{log}}\xrightarrow{f_0}G'_{\mr{log}}$ factors through $G'^{(Y')}_{\mr{log}}\hookrightarrow G'_{\mr{log}}$. Let $X$ and $X'$ be the character groups of the torus parts of $G$ and $G'$ respectively, let $f_{\mr{l}}:X'\rightarrow X$ be the map induced from $f_0$, and let
$$\tilde{f}_{\mr{d}}:\mc{H}om_{S^{\mr{cl}}_{\mr{\acute{E}t}}}(X,(\Gml/\Gm)_{S^{\mr{cl}}_{\mr{\acute{E}t}}})\rightarrow\mc{H}om_{S^{\mr{cl}}_{\mr{\acute{E}t}}}(X',(\Gml/\Gm)_{S^{\mr{cl}}_{\mr{\acute{E}t}}})$$
be the map induced from $f_{\mr{l}}$. By the definition of $G_{\mr{log}}^{(Y)}$ and $G'^{(Y')}_{\mr{log}}$, we are reduced to show the composition 
\begin{align*}
&\mc{H}om_{S^{\mr{cl}}_{\mr{\acute{E}t}}}(X,(\Gml/\Gm)_{S^{\mr{cl}}_{\mr{\acute{E}t}}})^{(Y)}\hookrightarrow \mc{H}om_{S^{\mr{cl}}_{\mr{\acute{E}t}}}(X,(\Gml/\Gm)_{S^{\mr{cl}}_{\mr{\acute{E}t}}})\xrightarrow{\tilde{f}_{\mr{d}}}  \\
 &\mc{H}om_{S^{\mr{cl}}_{\mr{\acute{E}t}}}(X',(\Gml/\Gm)_{S^{\mr{cl}}_{\mr{\acute{E}t}}})
\end{align*}
factors through $$\mc{H}om_{S^{\mr{cl}}_{\mr{\acute{E}t}}}(X',(\Gml/\Gm)_{S^{\mr{cl}}_{\mr{\acute{E}t}}})^{(Y')}\hookrightarrow \mc{H}om_{S^{\mr{cl}}_{\mr{\acute{E}t}}}(X',(\Gml/\Gm)_{S^{\mr{cl}}_{\mr{\acute{E}t}}}).$$ 
Let $<,>:X\times Y\rightarrow (\Gml/\Gm)_{S^{\mr{cl}}_{\mr{\acute{E}t}}}$ (resp. $<,>':X'\times Y'\rightarrow (\Gml/\Gm)_{S^{\mr{cl}}_{\mr{\acute{E}t}}}$) be the pairing associated to $M$ (resp. $M'$), then we have 
$$<f_{\mr{l}}(x'),y>=<x',f_{-1}(y)>'$$
for any $x'\in X', y\in Y$. For any $U\in (\mr{fs}/S)$, 
$$\varphi\in \mc{H}om_{S^{\mr{cl}}_{\mr{\acute{E}t}}}(X,(\Gml/\Gm)_{S^{\mr{cl}}_{\mr{\acute{E}t}}})^{(Y)}(U),$$
we need to show
$$\psi:=\varphi\circ f_{\mr{l}}\in \mc{H}om_{S^{\mr{cl}}_{\mr{\acute{E}t}}}(X',(\Gml/\Gm)_{S^{\mr{cl}}_{\mr{\acute{E}t}}})^{(Y')}(U).$$
For every $u\in U$ and every $x'\in X'_{\bar{u}}$, there exist $y_{u,x',1},y_{u,x',2}\in Y_{\bar{u}}$ such that $<f_{\mr{l}}(x'),y_{u,x',1}>|\,\varphi_{\bar{u}}(f_{\mr{l}}(x'))\,|<f_{\mr{l}}(x'),y_{u,x',2}>$. The relation can be rewritten as $$<x',f_{-1}(y_{u,x',1})>'|\, \psi_{\bar{u}}(x')\, |<x',f_{-1}(y_{u,x',2})>',$$
which implies that $\psi\in \mc{H}om_{S^{\mr{cl}}_{\mr{\acute{E}t}}}(X',(\Gml/\Gm)_{S^{\mr{cl}}_{\mr{\acute{E}t}}})^{(Y')}(U)$. This finishes the proof of part (3).
\end{proof}

\begin{rmk}\label{rmk1.2} 
Clearly, the image of $u:Y\rightarrow G_{\mr{log}}$ is contained in $G_{\mr{log}}^{(Y)}$.
\end{rmk}

We further assume that the pairing (\ref{eq1.1}) is non-degenerate (see \cite[7.1]{k-k-n2} and Remark \ref{rmk1.1} for the definition of non-degenerate pairings), then the two maps $X\rightarrow \mc{H}om_{S^{\mr{log}}_{\mr{fl}}}(Y,\Gml/\Gm)$ and $Y\rightarrow \mc{H}om_{S^{\mr{log}}_{\mr{fl}}}(X,\Gml/\Gm)$ associated to the pairing are both injective. Recall that in \cite[Def. 3.3. (1)]{k-k-n2} (resp. \cite[1.7]{k-k-n4}) a log abelian variety with constant degeneration (resp. weak log abelian variety with constant degeneration) over $S$ is defined to be a sheaf of abelian groups on $S^{\mr{cl}}_{\mr{\acute{E}t}}$ which is isomorphic to the quotient sheaf $(G_{\mr{log}}^{(Y)}/Y)_{S^{\mr{cl}}_{\mr{\acute{E}t}}}$ for a pointwise polarisable (resp. non-degenerate) log 1-motive $M=[Y\xrightarrow{u}G_{\mr{log}}]$. Here a log 1-motive is said to be non-degenerate if its associated pairing (\ref{eq1.1}) is non-degenerate. Since the polarisability implies the non-degeneracy, a log abelian variety with constant degeneration over $S$ is in particular a weak log abelian variety with constant degeneration over $S$.

\begin{thm}\label{thm1.1}
Let $A$ be a weak log abelian variety with constant degeneration over $S$. Suppose $A=(G_{\mr{log}}^{(Y)}/Y)_{S^{\mr{cl}}_{\mr{\acute{E}t}}}$ for a non-degenerate log 1-motive $M=[Y\xrightarrow{u}G_{\mr{log}}]$. Then 
\begin{enumerate}[(1)]
\item $A$ is a sheaf on $S^{\mr{log}}_{\mr{fl}}$;
\item $A=G_{\mr{log}}^{(Y)}/Y$, in other words $A$ fits into a canonical short exact sequence
\begin{equation}\label{eq1.3}
0\rightarrow Y\rightarrow G_{\mr{log}}^{(Y)}\rightarrow A\rightarrow 0
\end{equation}
in the category of sheaves of abelian groups on $S^{\mr{log}}_{\mr{fl}}$;
\item $A$ fits into a canonical short exact sequence
\begin{equation}\label{eq1.4}
0\rightarrow G\rightarrow A\rightarrow \mc{H}om_{S^{\mr{log}}_{\mr{fl}}}(X,\Gml/\Gm)^{(Y)}/Y\rightarrow 0
\end{equation}
in the category of sheaves of abelian groups on $S^{\mr{log}}_{\mr{fl}}$.
\end{enumerate}
\end{thm}
\begin{proof}
Part (2) follows from part (1). Since the log 1-motive $M$ is non-degenerate, the map $Y\rightarrow \mc{H}om_{S^{\mr{log}}_{\mr{fl}}}(X,\Gml/\Gm)^{(Y)}$ is injective. Then the short exact sequence in part (3) is induced from the short exact sequences (\ref{eq1.2}) and (\ref{eq1.3}). We are left with part (1). The proof of part (1) is similar to that for the log \'etale case in \cite[\S5]{k-k-n4}. 

Consider the short exact sequence $0\rightarrow Y\rightarrow \tilde{A}\rightarrow A\rightarrow 0$ of \cite[5.3]{k-k-n4}. Note that $\tilde{A}$ is nothing but $G_{\mr{log}}^{(Y)}$ in our situation, however we stick to the notation $\tilde{A}$ for the sake of coherence with \cite[5.3]{k-k-n4}. The argument showing that $\tilde{A}$ is a log \'etale sheaf, also shows that $\tilde{A}$ is a log flat sheaf, since representable functors are sheaves for the log flat topology by \cite[Thm. 5.2]{k-k-n4}. We have the canonical map $\delta:=m\circ \varepsilon_{\mr{fl}}:S^{\mr{log}}_{\mr{fl}}\xrightarrow{\varepsilon_{\mr{fl}}} S^{\mr{cl}}_{\mr{fl}}\xrightarrow{m} S^{\mr{cl}}_{\mr{\acute{E}t}}$ of sites. Applying $\delta^*$ and $\delta_*$ to $0\rightarrow Y\rightarrow \tilde{A}\rightarrow A\rightarrow 0$, we get a commutative diagram
$$
\xymatrix{
0\ar[r] &Y\ar[r]\ar[d]^{=} &\tilde{A}\ar[r]\ar[d]^{=} &A\ar[r]\ar[d] &0 \\
0\ar[r] &Y\ar[r]  &\tilde{A}\ar[r]  &\delta_*\delta^*A \ar[r] &R^1\delta_*Y
}$$
with exact rows, where the vertical maps are the ones given by the adjunction $(\delta^*,\delta_*)$. To prove that $A$ is a sheaf for the log flat topology, it is enough to show that the canonical map $A\rightarrow \delta_*\delta^*A$ is an isomorphism. This follows from the above commutative diagram with the help of the lemma below. 
\end{proof}

\begin{lem}\label{lem1.3}
The sheaf $R^1\delta_* Y$ is zero.
\end{lem}
\begin{proof}
Since $Y$ is \'etale locally isomorphic to a finite rank free abelian group, we are reduced to the case $Y=\Z$. Note that $Y$ is a smooth group scheme over $S$. The proof here is the same as the proof of \cite[Thm. 4.1]{kat2} (see also the proof of \cite[Thm. 3.12]{niz1}) except the very last part where the condition $G$ being affine is used. The reason why the proof there can be generalised to our case lies in the fact that $Y$ is \'etale over $S$. 

Now we start from \cite[the second half of page 22]{kat2} or \cite[the second last paragraph of page 524]{niz1}, since these two parallel parts are the very parts needed to be modified. Let $B$ be a strict local ring, $\hat{B}$ its completion, and let $\alpha\in H^1((\mr{Spec}B)_{\mr{fl}}^{\mr{log}},\Z)$ such that it vanishes in $H^1((\mr{Spec}\hat{B})_{\mr{fl}}^{\mr{log}},\Z)$. By fpqc descent, $\alpha$ is a class of a representable $\Z$-torsor over $\mr{Spec}B$ such that its structure morphism is \'etale. Since $B$ is a strict local ring, the torsor admits a section by \cite[Prop. 18.8.1]{egaIV-4}, so $\alpha$ is zero. It follows that \cite[Thm. 4.1]{kat2} also holds for the case $G=\Z$, so $R^1\varepsilon_{\mr{fl}*}\Z=0$. The Leray spectral sequence gives a short exact sequence $0\rightarrow R^1m_*\Z\rightarrow R^1\delta_*\Z\rightarrow m_*R^1\varepsilon_{\mr{fl}*}\Z$. The sheaf $R^1m_*\Z=0$ by \cite[Thm. 11.7]{gro1}, it follows that $R^1\delta_*\Z=0$.
\end{proof}
\begin{rmk}\label{rmk1.3}
Lemma \ref{lem1.3} can be viewed as a generalisation of Kato's theorem (see \cite[Thm. 4.1]{kat2} or \cite[Thm. 3.12]{niz1}) to  \'etale locally constant finitely generated torsion-free group schemes.
\end{rmk}

Now we give a reformulation of \cite[Thm. 7.4]{k-k-n2} in the context of the log flat topology. 

\begin{thm}\label{thm1.2}
Let $[Y\rightarrow G_{\mr{log}}]$ be a log 1-motive over $S$ of type $(X,Y)$ (see \cite[Def. 2.2]{k-k-n2}) such that the induced paring $X\times Y\rightarrow \Gml/\Gm$ is non-degenerate, and let $[X\rightarrow G^*_{\mr{log}}]$ be its dual. Let $A=G^{(Y)}_{\mr{log}}/Y$. Then we have:
\begin{enumerate}[(1)]
\item $\mc{E}xt_{S^{\mr{log}}_{\mr{fl}}}(A,\Z)\cong \mc{H}om_{S^{\mr{log}}_{\mr{fl}}}(Y,\Z)$;
\item the sheaf $\delta_*\mc{E}xt_{S^{\mr{log}}_{\mr{fl}}}(A,\Gm)$ fits into an exact sequence $$0\rightarrow G^*\rightarrow \delta_*\mc{E}xt_{S^{\mr{log}}_{\mr{fl}}}(A,\Gm)\rightarrow \mc{H}om_{S^{\mr{cl}}_{\mr{\acute{E}t}}}(A,R^1\delta_* \Gm);$$
\item $\mc{E}xt_{S^{\mr{log}}_{\mr{fl}}}(A,\Gml)\cong (G^*_{\mr{log}}/X)_{S^{\mr{cl}}_{\mr{\acute{E}t}}} \cong G^*_{\mr{log}}/X$;
\item $\mc{H}om_{S^{\mr{log}}_{\mr{fl}}}(A,\Z)=\mc{H}om_{S^{\mr{log}}_{\mr{fl}}}(A,\Gm)=\mc{H}om_{S^{\mr{log}}_{\mr{fl}}}(A,\Gml)=0$.
\end{enumerate}
\end{thm}
\begin{proof}
By Proposition \ref{prop1.1} and Theorem \ref{thm1.1}, the sheaves $\Z$, $\Gml$ and $A$ on $S^{\mr{cl}}_{\mr{\acute{E}t}}$ are also sheaves on $S^{\mr{log}}_{\mr{fl}}$. Then part (4) follows from \cite[Thm. 7.4 (4)]{k-k-n2} with the help of Lemma \ref{lem1.1}.

Before going to the rest of the proof, we first introduce two spectral sequences. Let $F_1$ (resp. $F_2$) be a sheaf on $S^{\mr{cl}}_{\mr{\acute{E}t}}$ (resp. $S^{\mr{log}}_{\mr{fl}}$), then we have 
$$\delta_*\mc{H}om_{S^{\mr{log}}_{\mr{fl}}}(\delta^*F_1,F_2)=\mc{H}om_{S^{\mr{cl}}_{\mr{\acute{E}t}}}(F_1,\delta_* F_2).$$
Let $\theta$ be the functor sending $F_2$ to $\delta_*\mc{H}om_{S^{\mr{log}}_{\mr{fl}}}(\delta^*F_1,F_2)=\mc{H}om_{S^{\mr{cl}}_{\mr{\acute{E}t}}}(F_1,\delta_* F_2)$, then we get two Grothendieck spectral sequences
\begin{equation}\label{eq1.5}
E_2^{p,q}=R^p\delta_* R^q\mc{H}om_{S^{\mr{log}}_{\mr{fl}}}(\delta^*F_1,-)\Rightarrow R^{p+q}\theta
\end{equation} 
and
\begin{equation}\label{eq1.6}
E_2^{p,q}=R^p\mc{H}om_{S^{\mr{cl}}_{\mr{\acute{E}t}}}(F_1,-) R^q\delta_* \Rightarrow R^{p+q}\theta .
\end{equation}

These two spectral sequences give two exact sequences
\begin{equation}\label{eq1.7}
\begin{split}
0&\rightarrow R^1\delta_*\mc{H}om_{S^{\mr{log}}_{\mr{fl}}}(\delta^*F_1,F_2)\rightarrow R^1\theta(F_2)\rightarrow \delta_*\mc{E}xt_{S^{\mr{log}}_{\mr{fl}}}(\delta^*F_1,F_2)  \\
&\rightarrow R^2\delta_*\mc{H}om_{S^{\mr{log}}_{\mr{fl}}}(\delta^*F_1,F_2)
\end{split}
\end{equation}
and
\begin{equation}\label{eq1.8}
0\rightarrow \mc{E}xt_{S^{\mr{cl}}_{\mr{\acute{E}t}}}(F_1,\delta_* F_2)\rightarrow R^1\theta (F_2)\rightarrow \mc{H}om_{S^{\mr{cl}}_{\mr{\acute{E}t}}}(F_1,R^1\delta_* F_2).
\end{equation}
Let $F_1=A$, and let $F_2$ be $\Z$, $\Gm$ or $\Gml$, part (4) together with Theorem \ref{thm1.1} and exact sequence (\ref{eq1.7}) imply $$R^1\theta (A)\cong \delta_*\mc{E}xt_{S^{\mr{log}}_{\mr{fl}}}(A,F_2),$$
so we get an exact sequence
$$0\rightarrow \mc{E}xt_{S^{\mr{cl}}_{\mr{\acute{E}t}}}(A,\delta_* F_2)\rightarrow \delta_*\mc{E}xt_{S^{\mr{log}}_{\mr{fl}}}(A,F_2)\rightarrow \mc{H}om_{S^{\mr{cl}}_{\mr{\acute{E}t}}}(A,R^1\delta_* F_2).$$
Since $\mc{E}xt_{S^{\mr{cl}}_{\mr{\acute{E}t}}}(A,\Gm)\cong G^*$ by \cite[Thm. 7.4 (2)]{k-k-n2}, the case $F_2=\Gm$ gives part (2). Since $R^1\delta \Z=0$ by Lemma \ref{lem1.3} and $\mc{E}xt_{S^{\mr{cl}}_{\mr{\acute{E}t}}}(A,\Z)\cong \mc{H}om_{S^{\mr{log}}_{\mr{fl}}}(Y,\Z)$ by \cite[Thm. 7.4 (1)]{k-k-n2}, the case $F_2=\Z$ gives part (1). The sheaf $R^1\delta_*\Gml$ equals zero by Kato's logarithmic Hilbert 90 \cite[Cor. 5.2]{kat2}. And we have $\mc{E}xt_{S^{\mr{cl}}_{\mr{\acute{E}t}}}(A,\Gml)\cong (G^*_{\mr{log}}/X)_{S^{\mr{cl}}_{\mr{\acute{E}t}}}$ by \cite[Thm. 7.4 (3)]{k-k-n2}. Then part (3) follows from the case $F_2=\Gml$.
\end{proof}

Let $A$ be a weak log abelian variety with constant degeneration over $S$, and let $M=[Y\rightarrow G_{\mr{log}}]$ be the log 1-motive of type $(X,Y)$ defining $A$. Then the paring $<,>:X\times Y\rightarrow \Gml/\Gm$ associated to $M$ is non-degenerate. Let $M^*=[X\rightarrow G^*_{\mr{log}}]$ be the dual log 1-motive of $M$, then the pairing associated to $M^*$ is the same (up to switching the positions of $X$ and $Y$) as the paring associated to $M$, hence it is automatically non-degenerate. If $A$ is further a log abelian variety with constant degeneration, i.e. the log 1-motive $M$ is pointwise polarisable, then $M^*$ is also pointwise polarisable.

\begin{defn}\label{defn1.1}
Let $A$ be a weak log abelian variety with constant degeneration (resp. log abelian variety with constant degeneration) over $S$. The dual weak log abelian variety with constant degeneration (resp. dual log abelian variety with constant degeneration) of $A$ is the weak log abelian variety with constant degeneration (resp. log abelian variety with constant degeneration) $G_{\mr{log}}^{*(X)}/X$ associated to the log 1-motive $M^*=[X\rightarrow G^*_{\mr{log}}]$. We denote the dual of $A$ by $A^*$.
\end{defn}

Let $\mr{WLAV}^{\mr{CD}}_S$ (resp. $\mr{LAV}^{\mr{CD}}_S$) denote the category of weak log abelian varieties with constant degeneration (resp. log abelian varieties with constant degeneration) over $S$. Then we have the following proposition.

\begin{prop}\label{prop1.5}
The association of $A^*$ to $A$ gives rise to a contravariant functor $$(-)^*:\mr{WLAV}^{\mr{CD}}_S\rightarrow \mr{WLAV}^{\mr{CD}}_S$$ which restricts to a contravariant functor $$(-)^*:\mr{LAV}^{\mr{CD}}_S\rightarrow \mr{LAV}^{\mr{CD}}_S.$$
Moreover the functor is a duality functor, i.e. there is a natural isomorphism from the identity functor to $(-)^{**}$.
\end{prop}
\begin{proof}
This follows from \cite[1.7]{k-k-n4}, \cite[Thm. 3.4]{k-k-n2}, and the corresponding duality theory of log 1-motives over $S$.
\end{proof}

\begin{rmk}\label{rmk1.4}
Given an abelian scheme $A$ over the underlying scheme of $S$, the dual abelian scheme $A^*$ can also be interpreted as $\mc{E}xt_{S^{\mr{cl}}_{\mr{fl}}}(A,\Gm)$. One may wonder if something similar happens in the case of (weak) log abelian varieties with constant degeneration. Note that in the log world, $\Gml$ plays the role of $\Gm$ in the non-log world. Part (3) of Theorem \ref{thm1.2} indicates that $\mc{E}xt_{S^{\mr{log}}_{\mr{fl}}}(A,\Gml)\cong G^*_{\mr{log}}/X$ is not $A^*=G^{*(X)}_{\mr{log}}/X$ but closely related to it.
\end{rmk}

The following is a partial reformulation of \cite[Thm. 7.3]{k-k-n2}.

\begin{thm}\label{thm1.3}
Let $X$ and $Y$ be two finitely generated free $\Z$-modules, and let $<,>:X\times Y\rightarrow \Gml/\Gm$ be a non-degenerate pairing on $S^{\mr{log}}_{\mr{fl}}$.

Let $G$ be a commutative group scheme over the underlying scheme of $S$ which is an extension of an abelian scheme $B$ by a torus $T$ over $S$. Assume that $X$ is the character group of $T$. Let $T_{\mr{log}}^{(Y)}=\mc{H}om_{S^{\mr{log}}_{\mr{fl}}}(X,\Gml)^{(Y)}\subset T_{\mr{log}}=\mc{H}om_{S^{\mr{log}}_{\mr{fl}}}(X,\Gml)$ (resp. $G_{\mr{log}}^{(Y)}\subset G_{\mr{log}}$) be the inverse image of $$\mc{H}om_{S^{\mr{log}}_{\mr{fl}}}(X,\Gml/\Gm)^{(Y)}\subset \mc{H}om_{S^{\mr{log}}_{\mr{fl}}}(X,\Gml/\Gm)\cong T_{\mr{log}}/T\cong G_{\mr{log}}/G.$$
\begin{enumerate}[(1)]
\item Let $H$ be a commutative group scheme over the underlying scheme of $S$. Then we have
$$\mc{H}om_{S^{\mr{log}}_{\mr{fl}}}(G_{\mr{log}}^{(Y)},H)\cong \mc{H}om_{S^{\mr{log}}_{\mr{fl}}}(B,H), \quad
\mc{H}om_{S^{\mr{log}}_{\mr{fl}}}(G_{\mr{log}}^{(Y)}/G,H)=0.$$
If further $H$ satisfies the condition $R^1\delta_* H=0$, we also have
$$\mc{E}xt_{S^{\mr{log}}_{\mr{fl}}}(G_{\mr{log}}^{(Y)},H)\cong \mc{E}xt_{S^{\mr{log}}_{\mr{fl}}}(B,H)$$
and
$$\mc{E}xt_{S^{\mr{log}}_{\mr{fl}}}(G_{\mr{log}}^{(Y)}/G,H)\cong \mc{H}om_{S^{\mr{log}}_{\mr{fl}}}(T,H).$$
In particular, since $R^1\delta_*\Z=0$, we have 
$$\mc{E}xt_{S^{\mr{log}}_{\mr{fl}}}(G_{\mr{log}}^{(Y)},\Z)\cong \mc{E}xt_{S^{\mr{log}}_{\mr{fl}}}(B,\Z)=0$$
and
$$\mc{E}xt_{S^{\mr{log}}_{\mr{fl}}}(G_{\mr{log}}^{(Y)}/G,\Z)\cong \mc{H}om_{S^{\mr{log}}_{\mr{fl}}}(T,\Z)=0.$$
\item We have $\mc{H}om_{S^{\mr{log}}_{\mr{fl}}}(T_{\mr{log}}^{(Y)},\Gml)\cong X, \mc{H}om_{S^{\mr{log}}_{\mr{fl}}}(T_{\mr{log}}^{(Y)}/T,\Gml)=0,\\ 
\mc{E}xt_{S^{\mr{log}}_{\mr{fl}}}(T_{\mr{log}}^{(Y)}/T,\Gml)=0$, and $\delta_*\mc{E}xt_{S^{\mr{log}}_{\mr{fl}}}(T_{\mr{log}}^{(Y)},\Gml)\subset R^2\delta_* X$.
\item Let $G'$ be another commutative group scheme over $S$ which is an extension of an abelian scheme $B'$ by a torus $T'$ over $S$. Let $X'$ be the character group of $T'$. Then we have 
$$\mc{H}om_{S^{\mr{log}}_{\mr{fl}}}(G,G')\xrightarrow{\cong} \mc{H}om_{S^{\mr{log}}_{\mr{fl}}}(G_{\mr{log}}^{(Y)},G'_{\mr{log}})$$
and
$$\mc{H}om_{S^{\mr{log}}_{\mr{fl}}}(X',X)\otimes_{\Z}\Q\xrightarrow{\cong} \mc{H}om_{S^{\mr{log}}_{\mr{fl}}}(G_{\mr{log}}^{(Y)}/G,G'_{\mr{log}}/G').$$
\end{enumerate}
\end{thm}

Our proof of Theorem \ref{thm1.3} follows the structure of the proof \cite[7.17-7.21]{k-k-n2} of \cite[Thm. 7.3]{k-k-n2}. Firstly we state a lemma which almost corresponds to \cite[7.17]{k-k-n2}.

\begin{lem}\label{lem1.4}
Let the notation be as in Theorem \ref{thm1.3}. Then the following hold.
\begin{enumerate}[(1)]
\item $\mc{H}om_{S^{\mr{log}}_{\mr{fl}}}(T_{\mr{log}}^{(Y)},H)=0;$
\item If $R^1\delta_* H=0$, then $\mc{E}xt_{S^{\mr{log}}_{\mr{fl}}}(T_{\mr{log}}^{(Y)},H)=0$;
\item $\delta_* \mc{E}xt_{S^{\mr{log}}_{\mr{fl}}}(T_{\mr{log}}^{(Y)},\Gml)\subset R^2\delta_* X$.
\end{enumerate}
\end{lem}
\begin{proof}
Part (1) follows from the corresponding result in the classical \'etale topology of \cite[7.17]{k-k-n2}, with the help of Lemma \ref{lem1.1}. 

We use the exact sequences (\ref{eq1.7}) and (\ref{eq1.8}) from the proof of Theorem \ref{thm1.2} to investigate the sheaves $\mc{E}xt_{S^{\mr{log}}_{\mr{fl}}}(T_{\mr{log}}^{(Y)},H)$ and $\mc{E}xt_{S^{\mr{log}}_{\mr{fl}}}(T_{\mr{log}}^{(Y)},\Gml)$. If $R^1\delta_* H=0$, then the vanishing of $\mc{H}om_{S^{\mr{log}}_{\mr{fl}}}(T_{\mr{log}}^{(Y)},H)$ and $\mc{E}xt_{S^{\mr{cl}}_{\mr{\acute{E}t}}}(T_{\mr{log}}^{(Y)},H)$ implies part (2) via (\ref{eq1.7}) and (\ref{eq1.8}). We have $$\mc{E}xt_{S^{\mr{cl}}_{\mr{\acute{E}t}}}(T_{\mr{log}}^{(Y)},\Gml)=0$$ by \cite[7.17]{k-k-n2} and $R^1\delta_*\Gml=0$ by Kato's logarithmic Hilbert 90. We also have
$$\mc{H}om_{S^{\mr{log}}_{\mr{fl}}}(T_{\mr{log}}^{(Y)},\Gml)=\mc{H}om_{S^{\mr{cl}}_{\mr{\acute{E}t}}}(T_{\mr{log}}^{(Y)},\Gml)\cong X$$
by \cite[Thm. 7.3 (2)]{k-k-n2}. Then the inclusion of part (3) follows from (\ref{eq1.7}) and (\ref{eq1.8}). 
\end{proof}

\begin{rmk}\label{rmk1.5}
In the proof of part (3) of Lemma \ref{lem1.4}, since
$$\mc{E}xt_{S^{\mr{cl}}_{\mr{\acute{E}t}}}(T_{\mr{log}}^{(Y)},\Gml)=R^1\delta_*\Gml=0,$$
we must have $R^1\delta_* X=0$ by the exact sequences (\ref{eq1.7}) and (\ref{eq1.8}). This gives rise to an alternative proof of Lemma \ref{lem1.3}.
\end{rmk}

\begin{proof}[Proof of Theorem \ref{thm1.3}:]
Firstly we prove part (1). The isomorphism $$\mc{H}om_{S^{\mr{log}}_{\mr{fl}}}(G_{\mr{log}}^{(Y)},H)\cong \mc{H}om_{S^{\mr{log}}_{\mr{fl}}}(B,H)$$ comes from the corresponding isomorphism of \cite[Thm. 7.3 (1)]{k-k-n2} for the classical \'etale topology with the help of Lemma \ref{lem1.1}. The short exact sequence $0\rightarrow T\rightarrow T_{\mr{log}}^{(Y)}\rightarrow G_{\mr{log}}^{(Y)}/G\rightarrow 0$ gives rise to a long exact sequence
\begin{align*}
0& \rightarrow \mc{H}om_{S^{\mr{log}}_{\mr{fl}}}(G_{\mr{log}}^{(Y)}/G,H)\rightarrow \mc{H}om_{S^{\mr{log}}_{\mr{fl}}}(T_{\mr{log}}^{(Y)},H)\rightarrow \mc{H}om_{S^{\mr{log}}_{\mr{fl}}}(T,H)  \\
 & \rightarrow \mc{E}xt_{S^{\mr{log}}_{\mr{fl}}}(G_{\mr{log}}^{(Y)}/G,H)\rightarrow \mc{E}xt_{S^{\mr{log}}_{\mr{fl}}}(T_{\mr{log}}^{(Y)},H).
\end{align*}
Since we have $\mc{H}om_{S^{\mr{log}}_{\mr{fl}}}(T_{\mr{log}}^{(Y)},H)=0$ by part (1) of Lemma \ref{lem1.4}, it follows that $\mc{H}om_{S^{\mr{log}}_{\mr{fl}}}(G_{\mr{log}}^{(Y)}/G,H)=0$. Assuming $R^1\delta_* H=0$, we get $\mc{E}xt_{S^{\mr{log}}_{\mr{fl}}}(T_{\mr{log}}^{(Y)},H)=0$ by part (2) of Lemma \ref{lem1.4}. Hence we have
$$\mc{E}xt_{S^{\mr{log}}_{\mr{fl}}}(G_{\mr{log}}^{(Y)},H)\cong \mc{E}xt_{S^{\mr{log}}_{\mr{fl}}}(B,H), \quad\mc{E}xt_{S^{\mr{log}}_{\mr{fl}}}(G_{\mr{log}}^{(Y)}/G,H)\cong \mc{H}om_{S^{\mr{log}}_{\mr{fl}}}(T,H).$$

Next we prove part (2). The exact sequence $0\rightarrow T\rightarrow T_{\mr{log}}^{(Y)}\rightarrow T_{\mr{log}}^{(Y)}/T\rightarrow 0$ gives rise to a long exact sequence
\begin{align*}
0& \rightarrow \mc{H}om_{S^{\mr{log}}_{\mr{fl}}}(T_{\mr{log}}^{(Y)}/T,\Gml)\rightarrow \mc{H}om_{S^{\mr{log}}_{\mr{fl}}}(T_{\mr{log}}^{(Y)},\Gml)  \\
&\xrightarrow{\alpha} \mc{H}om_{S^{\mr{log}}_{\mr{fl}}}(T,\Gml) \rightarrow \mc{E}xt_{S^{\mr{log}}_{\mr{fl}}}(T_{\mr{log}}^{(Y)}/T,\Gml).
\end{align*}
Since the map $\mc{H}om_{S^{\mr{cl}}_{\mr{\acute{E}t}}}(T_{\mr{log}}^{(Y)},\Gml)\rightarrow \mc{H}om_{S^{\mr{cl}}_{\mr{\acute{E}t}}}(T,\Gml)$ is canonically identical to the identity map $1_{X}:X\rightarrow X$ by \cite[7.20]{k-k-n2}, so is the map $\alpha$. Hence we have $\mc{H}om_{S^{\mr{log}}_{\mr{fl}}}(T_{\mr{log}}^{(Y)}/T,\Gml)=0$. Since 
$R^1\delta_*\Gml=0$
and $\mc{E}xt_{S^{\mr{cl}}_{\mr{\acute{E}t}}}((T_{\mr{log}}^{(Y)}/T)_{S^{\mr{cl}}_{\mr{\acute{E}t}}},\Gml)=0$ by \cite[Thm. 7.3 (2)]{k-k-n2}, we get  $\mc{E}xt_{S^{\mr{log}}_{\mr{fl}}}(T_{\mr{log}}^{(Y)}/T,\Gml)=0$ by the exact sequences (\ref{eq1.7}) and (\ref{eq1.8}). The inclusion $\delta_*\mc{E}xt_{S^{\mr{log}}_{\mr{fl}}}(T_{\mr{log}}^{(Y)},\Gml)\subset R^2\delta_* X$ is just part (3) of Lemma \ref{lem1.4}. The isomorphism $\mc{H}om_{S^{\mr{log}}_{\mr{fl}}}(T_{\mr{log}}^{(Y)},\Gml)\cong X$ has been proved in the proof of Lemma \ref{lem1.4}.

At last, we show part (3). The isomorphism $$\mc{H}om_{S^{\mr{log}}_{\mr{fl}}}(G,G')\xrightarrow{\cong} \mc{H}om_{S^{\mr{log}}_{\mr{fl}}}(G_{\mr{log}}^{(Y)},G'_{\mr{log}})$$ comes from the corresponding isomorphism of \cite[Thm. 7.3 (3)]{k-k-n2} with the help of Lemma \ref{lem1.1}. The short exact sequence $$0\rightarrow T\rightarrow T^{(Y)}_{\mr{log}}\rightarrow G^{(Y)}_{\mr{log}}/G\rightarrow 0$$ gives an exact sequence 
\begin{align*}
 0\rightarrow &\mc{H}om_{S^{\mr{log}}_{\mr{fl}}}(G_{\mr{log}}^{(Y)}/G,G'_{\mr{log}}/G')\rightarrow \mc{H}om_{S^{\mr{log}}_{\mr{fl}}}(T_{\mr{log}}^{(Y)},G'_{\mr{log}}/G')  \\
 \rightarrow &\mc{H}om_{S^{\mr{log}}_{\mr{fl}}}(T,G'_{\mr{log}}/G').
\end{align*}
The sheaf $\mc{H}om_{S^{\mr{log}}_{\mr{fl}}}(T,G'_{\mr{log}}/G')$ is zero by part (2) of Lemma \ref{lem1.2}, hence we are reduced to compute $\mc{H}om_{S^{\mr{log}}_{\mr{fl}}}(T_{\mr{log}}^{(Y)},G'_{\mr{log}}/G')$. We have 
\begin{align*}
\delta_*\mc{H}om_{S^{\mr{log}}_{\mr{fl}}}(T_{\mr{log}}^{(Y)},G'_{\mr{log}}/G')&=\mc{H}om_{S^{\mr{cl}}_{\mr{\acute{E}t}}}(T_{\mr{log}}^{(Y)},\delta_*(G'_{\mr{log}}/G'))  \\
&=\mc{H}om_{S^{\mr{cl}}_{\mr{\acute{E}t}}}(T_{\mr{log}}^{(Y)},(G'_{\mr{log}}/G')_{S^{\mr{cl}}_{\mr{\acute{E}t}}}\otimes_{\Z}\Q)
\end{align*}
where the second equality comes from part (1) of Lemma \ref{lem1.2}. By \cite[Thm. 7.3 (3)]{k-k-n2}, 
\begin{align*}
\mc{H}om_{S^{\mr{cl}}_{\mr{\acute{E}t}}}(X',X)\xrightarrow{\cong} 
&\mc{H}om_{S^{\mr{cl}}_{\mr{\acute{E}t}}}((G_{\mr{log}}^{(Y)}/G)_{S^{\mr{cl}}_{\mr{\acute{E}t}}},(G'_{\mr{log}}/G')_{S^{\mr{cl}}_{\mr{\acute{E}t}}})    \\
\xrightarrow{\cong} &\mc{H}om_{S^{\mr{cl}}_{\mr{\acute{E}t}}}(T_{\mr{log}}^{(Y)},(G'_{\mr{log}}/G')_{S^{\mr{cl}}_{\mr{\acute{E}t}}}).
\end{align*}
It follows that $$\mc{H}om_{S^{\mr{log}}_{\mr{fl}}}(T_{\mr{log}}^{(Y)},G'_{\mr{log}}/G')\cong \mc{H}om_{S^{\mr{cl}}_{\mr{\acute{E}t}}}(X',X)\otimes_{\Z}\Q=\mc{H}om_{S^{\mr{log}}_{\mr{fl}}}(X',X)\otimes_{\Z}\Q.$$
\end{proof}

Since weak log abelian varieties with constant degeneration are defined in terms of log 1-motives, it is natural to try to relate every aspect of weak log abelian varieties with constant degeneration to the corresponding aspect of log 1-motives. In particular, we are keen on the relation on the homomorphisms. The following theorem is a combination of \cite[Thm. 3.4]{k-k-n2} and \cite[1.7]{k-k-n4}.

\begin{thm}\label{thm1.5}
The functor $[Y\rightarrow G_{\mr{log}}]\mapsto (G^{(Y)}_{\mr{log}}/Y)_{S^{\mr{cl}}_{\mr{\acute{Et}}}}=G^{(Y)}_{\mr{log}}/Y$ induces an equivalence from the category of non-degenerate log 1-motives (resp. pointwise polarisable log 1-motives) over $S$ to that of weak log abelian varieties with constant degeneration (resp. log abelian varieties with constant degeneration) over $S$.
\end{thm}
\begin{proof}
See \cite[\S 8]{k-k-n2}.
\end{proof}

Let $f:A\rightarrow A'$ be a homomorphism between two weak log abelian varieties with constant degeneration over $S$, and let $M=[Y\rightarrow G_{\mr{log}}],M'=[Y'\rightarrow G'_{\mr{log}}]$ be the log 1-motives defining $A$ and $A'$ respectively. By Theorem \ref{thm1.5}, $f$ comes from a homomorphism from $M$ to $M'$, and we denote it by $(f_{-1},f_0)$. The proof of \cite[Thm. 8.1]{k-k-n2} actually shows that $f_0$ comes from a unique homomorphism from $G$ to $G'$, and we denote it by $f_{\mr{c}}$ by convention \footnote{Here the subscript c stands for connected.}. The homomorphism $f_{\mr{c}}$ can also be obtained from the following diagram
$$\xymatrix{
0\ar[r] &G\ar[r]\ar@{-->}[d]^{f_{\mr{c}}} &A\ar[r]\ar[d]^f &\mc{Q}/Y\ar[r]\ar@{-->}[d]^{f_{\mr{d}}} &0 \\
0\ar[r] &G'\ar[r]  &A'\ar[r]  &\mc{Q}'/Y'\ar[r] &0
}$$
with exact rows, together with the vanishing of $\mr{Hom}_{S^{\mr{log}}_{\mr{fl}}}(G,\mc{Q}'/Y')$ (see Lemma \ref{lem1.5} below). Here the exact rows come from part (3) of Theorem \ref{thm1.1}, and $\mc{Q}$ (resp. $\mc{Q}'$) denotes the sheaf $\mc{H}om_{S^{\mr{log}}_{\mr{fl}}}(X,\Gml/\Gm)^{(Y)}$ (resp. $\mc{H}om_{S^{\mr{log}}_{\mr{fl}}}(X',\Gml/\Gm)^{(Y')}$). Furthermore the diagram gives a homomorphism $\mc{Q}/Y\rightarrow \mc{Q}'/Y'$ which we denote by $f_{\mr{d}}$ \footnote{Here the subscript d stands for discrete.}. The procedure of getting $f$ from $(f_{-1},f_0)$ also gives a homomorphism $\tilde{A}:=G_{\mr{log}}^{(Y)}\rightarrow G'^{(Y')}_{\mr{log}}=:\tilde{A}'$, which we denote by $\tilde{f}$. The homomorphism $\tilde{f}$ induces a homomorphism $\mc{Q}\rightarrow \mc{Q}'$ which we denote by $\tilde{f}_{\mr{d}}$ \footnote{Here the symbol $\tilde{}$ for $\tilde{f}$ (resp. $\tilde{f}_{\mr{d}}$) stands for the lifting of $f$ (resp. $f_{\mr{d}}$) to the ``universal coverings'', and $\tilde{A}$ (resp. $\mc{Q}$) could be thought of as the ``universal covering'' of $A$ (resp. $\mc{Q}/Y$).}.

\begin{lem}\label{lem1.5}
We have $\mr{Hom}_{S^{\mr{log}}_{\mr{fl}}}(G,\mc{Q}'/Y')=0$.
\end{lem}
\begin{proof}
As before, let $\delta: S^{\mr{log}}_{\mr{fl}}\rightarrow S^{\mr{cl}}_{\mr{\acute{E}t}}$ be the canonical map between these two sites. By part (3) of Lemma \ref{lem1.2}, we have a canonical isomorphism $\delta_*\mc{Q}'\cong \mc{Q}'_{\mr{cl}}\otimes_{\Z}\Q$ with $\mc{Q}'_{\mr{cl}}:=\mc{H}om_{S^{\mr{cl}}_{\mr{\acute{E}t}}}(X',(\Gml/\Gm)_{S^{\mr{cl}}_{\mr{\acute{E}t}}})^{(Y')}$. We also have $\delta_*(\mc{Q}'/Y')=(\delta_*\mc{Q}'/Y')_{S^{\mr{cl}}_{\mr{\acute{E}t}}}$ by Lemma \ref{lem1.3}. Hence 
\begin{align*}
\mr{Hom}_{S^{\mr{log}}_{\mr{fl}}}(G,\mc{Q}'/Y')  
&=\mr{Hom}_{S^{\mr{cl}}_{\mr{\acute{E}t}}}(G,\delta_*(\mc{Q}'/Y'))  \\
&=\mr{Hom}_{S^{\mr{cl}}_{\mr{\acute{E}t}}}(G,(\delta_*\mc{Q}'/Y')_{S^{\mr{cl}}_{\mr{\acute{E}t}}})  \\
&=\mr{Hom}_{S^{\mr{cl}}_{\mr{\acute{E}t}}}(G,((\mc{Q}'_{\mr{cl}}\otimes_{\Z}\Q)/Y')_{S^{\mr{cl}}_{\mr{\acute{E}t}}}).
\end{align*}
The group $\mr{Hom}_{S^{\mr{cl}}_{\mr{\acute{E}t}}}(G,(\mc{Q}'_{\mr{cl}}\otimes_{\Z}\Q/Y')_{S^{\mr{cl}}_{\mr{\acute{E}t}}})$ equals zero by the same reason as in \cite[9.2]{k-k-n2}. Hence $\mr{Hom}_{S^{\mr{log}}_{\mr{fl}}}(G,\mc{Q}'/Y')$ vanishes.
\end{proof}

For practical reason, we state the following proposition, which is nothing else than a tedious summary of various maps constructed out of $f:A\rightarrow A'$.

\begin{prop}\label{prop1.6}
Let $f:A\rightarrow A'$ be a homomorphism of weak log abelian varieties with constant degeneration over $S$. Then $f$ induces the following four commutative diagrams
\begin{equation}\label{eq1.9}
\xymatrix{
0\ar[r] &Y\ar[r]\ar[d]^{f_{-1}} &G^{(Y)}_{\mr{log}}\ar[r]\ar[d]^{\tilde{f}} &A\ar[r]\ar[d]^f &0 \\
0\ar[r] &Y'\ar[r]  &G'^{(Y')}_{\mr{log}}\ar[r]  &A'\ar[r] &0
}
\end{equation}
\begin{equation}\label{eq1.10}
\xymatrix{
0\ar[r] &G\ar[r]\ar[d]^{f_{\mr{c}}} &A\ar[r]\ar[d]^f &\mc{Q}/Y\ar[r]\ar[d]^{f_{\mr{d}}} &0 \\
0\ar[r] &G'\ar[r]  &A'\ar[r]  &\mc{Q}'/Y'\ar[r] &0
}
\end{equation}
\begin{equation}\label{eq1.11}
\xymatrix{
0\ar[r] &G\ar[r]\ar[d]^{f_{\mr{c}}} &G^{(Y)}_{\mr{log}}\ar[r]\ar[d]^{\tilde{f}} &\mc{Q}\ar[r]\ar[d]^{\tilde{f}_{\mr{d}}} &0 \\
0\ar[r] &G'\ar[r]  &G'^{(Y')}_{\mr{log}}\ar[r]  &\mc{Q}'\ar[r] &0
}
\end{equation}
\begin{equation}\label{eq1.12}
\xymatrix{
0\ar[r] &Y\ar[r]\ar[d]^{f_{-1}} &\mc{Q}\ar[r]\ar[d]^{\tilde{f}_{\mr{d}}} &\mc{Q}/Y\ar[r]\ar[d]^{f_{\mr{d}}} &0 \\
0\ar[r] &Y'\ar[r]  &\mc{Q}'\ar[r]  &\mc{Q}'/Y'\ar[r] &0
}
\end{equation}
with exact rows.
\end{prop}

\section{Isogeny}\label{sec3}

In this section we study log abelian varieties over a log point. Note that in this case, log abelian varieties are necessarily log abelian varieties with constant degeneration by \cite[Thm. 4.6 (2)]{k-k-n2}.

Let $k$ be a field, and $S=(\mr{Spec}\,k,M_S)$ an fs log point with log structure induced by a chart $P\rightarrow k$, where $P$ is a sharp fs monoid such that $P\rightarrow (M_S/\mc{O}_S^{\times})_{\bar{x}}$ is an isomorphism. Here $x$ denotes the underlying point of $S$ and $\bar{x}$ denotes a geometric point above $x$. Let $(\mr{fs}/S)$ be the category of fs log schemes over $S$, and log schemes in this section will always be fs log schemes unless otherwise stated. Let $S^{\mr{log}}_{\mr{fl}}$ (resp. $S^{\mr{cl}}_{\mr{fl}}$) be the log flat (resp. classical flat) site on $(\mr{fs}/S)$, and let $S^{\mr{log}}_{\mr{\acute{E}t}}$ (resp. $S^{\mr{cl}}_{\mr{\acute{E}t}}$) be the log \'etale (resp. classical \'etale) site on $(\mr{fs}/S)$.

\subsection{Isogeny}
Firstly we show a few properties of the category $(\mr{fin}/S)_r$. See Definition \ref{defn-app.1} in the Appendix for the categories $(\mr{fin}/S)_f$, $(\mr{fin}/S)_r$, $(\mr{fin}/S)_d$ and $(\mr{fin}/S)_c$.

\begin{lem}\label{lem2.1}
Let $F_1\in (\mr{fin}/S)_c$ and let $F_2$ be a subobject of $F_1$ in $(\mr{fin}/S)_r$. Then we have $F_2\in (\mr{fin}/S)_c$. In other words, the category $(\mr{fin}/S)_c$ is closed under subobjects in $(\mr{fin}/S)_r$.
\end{lem}
\begin{proof}
If the field $k$ is of characteristic zero, then $F_2$ is a log flat locally constant sheaf of finite abelian groups. Hence it corresponds to a $\pi_1^{\mr{log}}(S)$-module by Theorem \ref{thm-app.2}. Here $\pi_1^{\mr{log}}(S)$ denotes the logarithmic fundamental group of $S$ with respect to some log geometric point of $S$, see \cite[\S 4]{ill1} for its definition. We have a short exact sequence
$$1\rightarrow I^{\mr{log}}(S)\rightarrow \pi_1^{\mr{log}}(S)\xrightarrow{\mr{forg}} \pi_1(S)\rightarrow 1$$
from \cite[(4.7.1)]{ill1}, where forg is the map to the classical \'etale fundamental group of $S$ induced by the canonical map from the log \'etale site to the classical \'etale site, and the kernel $I^{\mr{log}}(S)$ of forg is called the log inertia group of $S$. Since $F_1\in (\mr{fin}/S)_c$, the action of $I^{\mr{log}}(S)$ on $F_1$ is trivial. It follows that the action of $I^{\mr{log}}(S)$ on the subobject $F_2$ of $F_1$ is also trivial. Hence we get $F_2\in (\mr{fin}/S)_c$.

Now we are left with the case that the field $k$ is of positive characteristic. Let $0\rightarrow F_i^{\circ}\rightarrow F_i\rightarrow F_i^{\mr{et}}\rightarrow 0$ be the connected-\'etale short exact sequence of $F_i$, see Lemma \ref{lem-app.1}. It is obvious that we have $F_1^{\circ},F_1^{\mr{et}}\in (\mr{fin}/S)_c$. We have $F_2^{\mr{et}}\in (\mr{fin}/S)_r$ and $F_2^{\circ}\in (\mr{fin}/S)_c$ by Proposition \ref{prop-app.2}. The inclusion $i:F_2\hookrightarrow F_1$ gives a commutative diagram
$$\xymatrix{
0\ar[r] &F_2^{\circ}\ar[r]\ar@{^{(}->}[d]^{i^{\circ}} &F_2\ar[r]\ar@{^{(}->}[d]^i &F_2^{\mr{et}}\ar[r]\ar@{^{(}->}[d]^{i^{\mr{et}}} &0 \\
0\ar[r] &F_1^{\circ}\ar[r]  &F_1\ar[r]  &F_1^{\mr{et}} \ar[r] &0
}$$
with exact rows and injective vertical homomorphisms. Then we have $F_2^{\mr{et}}\in (\mr{fin}/S)_c$ by applying the same argument as in the characteristic zero case to $F_2^{\mr{et}}\subset F_1^{\mr{et}}$. Let $E$ be the pullback of the extension $F_1$ along $i^{\mr{et}}:F_2^{\mr{et}}\hookrightarrow F_1^{\mr{et}}$, then we have a commutative diagram
$$\xymatrix{
0\ar[r] &F_2^{\circ}\ar[r]\ar@{^{(}->}[d]^{i^{\circ}} &F_2\ar[r]\ar@{^{(}->}[d] &F_2^{\mr{et}}\ar[r]\ar@{=}[d] &0 \\
0\ar[r] &F_1^{\circ}\ar[r]\ar@{=}[d] &E\ar[r]\ar@{^{(}->}[d] &F_2^{\mr{et}}\ar[r]\ar@{^{(}->}[d]^{i^{\mr{et}}} &0 \\
0\ar[r] &F_1^{\circ}\ar[r]  &F_1\ar[r]  &F_1^{\mr{et}} \ar[r] &0
}$$
with exact rows. Note that $E$ lies in $(\mr{fin}/S)_c$, and $E$ is also the pushout of $F_2$ along $i^{\circ}:F_2^{\circ}\hookrightarrow F_1^{\circ}$. Now we make use of Kato's 
classification theorem Theorem \ref{thm-app.3}. Note that the functor $\Phi$ of Theorem \ref{thm-app.3} is compatible with pushout along the second argument, hence we have a commutative diagram
$$\renewcommand{\arraystretch}{1.3}
\begin{array}[c]{ccccc}
\mathfrak{Ext}_{S^{\mr{cl}}_{\mr{fl}}}(F_2^{\mr{et}},F_2^{\circ})&\times &\mathfrak{Hom}_{S^{\mr{cl}}_{\mr{fl}}}(F_2^{\mr{et}}(1),F_2^{\circ})\otimes P^{\mr{gp}} &\stackrel{\simeq}\rightarrow &\mathfrak{Ext}_{S^{\mr{log}}_{\mr{fl}}}(F_2^{\mr{et}},F_2^{\circ})\\

\downarrow\scriptstyle{i^{\circ}_*}&&\downarrow\scriptstyle{i^{\circ}_*}&&\downarrow\scriptstyle{i^{\circ}_*}\\

\mathfrak{Ext}_{S^{\mr{cl}}_{\mr{fl}}}(F_2^{\mr{et}},F_1^{\circ})&\times&\mathfrak{Hom}_{S^{\mr{cl}}_{\mr{fl}}}(F_2^{\mr{et}}(1),F_1^{\circ})\otimes P^{\mr{gp}}&\stackrel{\simeq}\rightarrow&\mathfrak{Ext}_{S^{\mr{log}}_{\mr{fl}}}(F_2^{\mr{et}},F_1^{\circ}) 
\end{array}$$
with rows equivalences of categories. Let $[E]\in\mathfrak{Ext}_{S^{\mr{log}}_{\mr{fl}}}(F_2^{\mr{et}},F_1^{\circ})$ (resp. $[F_2]\in \mathfrak{Ext}_{S^{\mr{log}}_{\mr{fl}}}(F_2^{\mr{et}},F_2^{\circ})$) denote the class represented by $E$ (resp. $F_2$). With the help of the above commutative diagram, $[E]\in\mathfrak{Ext}_{S^{\mr{cl}}_{\mr{fl}}}(F_2^{\mr{et}},F_1^{\circ})$ implies $[F_2]\in \mathfrak{Ext}_{S^{\mr{cl}}_{\mr{fl}}}(F_2^{\mr{et}},F_2^{\circ})$.
\end{proof}

\begin{prop}\label{prop2.1}
\begin{enumerate}[(1)]
\item The category $(\mr{fin}/S)_f$ is abelian.
\item The category $(\mr{fin}/S)_r$ is a weak Serre subcategory of $(\mr{fin}/S)_f$.
\item The category $(\mr{fin}/S)_c$ is closed under subobjects and quotient objects in $(\mr{fin}/S)_r$, but not closed under extensions in $(\mr{fin}/S)_r$. 
\end{enumerate}
\end{prop}
\begin{proof}
We first show part (1). Let $f:F\rightarrow F'$ be in $(\mr{fin}/S)_f$, and $U\rightarrow S$ a log flat cover such that both $F_U$ and $F'_U$ lie in $(\mr{fin}/U)_c$. Since $S$ is a log point, we may shrink $U$ such that its underlying scheme is affine (in particular quasi-compact). For each positive integer $n$, let $S_n:=S\times_{\mr{Spec}\Z[P]}\mr{Spec}\Z[P^{1/n}]$ endowed with the log structure associated to $P^{1/n}\rightarrow \mc{O}_{S_n}$.  By \cite[Prop. 2.7 (2)]{kat2} or \cite[Cor. 2.16]{niz1}\footnote{The proof of \cite[Prop. 2.7]{kat2} hasn't been given in the very preprint. One may refer to \cite[Cor. 2.16]{niz1} for the proof, however the statement might have missed the quasi-compact assumption. Nevertheless the proof works under the quasi-compact assumption.}, there exists a log flat cover $V\rightarrow U$ and some positive integer $n_0$ such that $W:=V\times_{S}S_{n_0}\rightarrow S_{n_0}$ is classically flat. Since both $F_W$ and $F'_W$ are represented by classical finite flat group schemes, so are $F_{S_{n_0}}$ and $F'_{S_{n_0}}$ by classical flat descent theory. Note that the underlying scheme of $S_{n_0}$ is artinian. Since the category of commutative finite flat group schemes over an artinian base is abelian, part (1) follows.

To show part (2), we need to check that $(\mr{fin}/S)_r$ is closed under kernels, cokernels and extensions. The closedness under kernels is trivial. The closedness under extensions is given by Proposition \ref{prop-app.1}. We are left to show the closedness under cokernels. It suffices to show that, for any short exact sequence $0\rightarrow F_1\rightarrow F_2\rightarrow F_3\rightarrow 0$ with $F_1,F_2\in (\mr{fin}/S)_r$, we must have $F_3\in (\mr{fin}/S)_r$. Let $0\rightarrow F_i^{\circ}\rightarrow F_i\rightarrow F_i^{\mr{et}}\rightarrow 0$ be the connected-\'etale short exact sequence of $F_i$, then we have the following commutative diagram
$$\xymatrix{
&0\ar[d]&0\ar[d]&0\ar[d]   \\
0\ar[r] &F_1^{\circ}\ar[r]\ar[d] &F_1\ar[r]\ar[d] &F_1^{\mr{et}}\ar[r]\ar[d] &0  \\
0\ar[r] &F_2^{\circ}\ar[r]\ar[d] &F_2\ar[r]\ar[d] &F_2^{\mr{et}}\ar[r]\ar[d] &0  \\
0\ar[r] &F_3^{\circ}\ar[r]\ar[d] &F_3\ar[r]\ar[d] &F_3^{\mr{et}}\ar[r]\ar[d] &0  \\
&0&0&0
}$$
with exact rows and columns. Both $F_1^{\circ}$ and $F_2^{\circ}$ lie in $(\mr{fin}/S)_c$, so is $F_3^{\circ}$. Hence $F_3\in (\mr{fin}/S)_r$ by Proposition \ref{prop-app.2}.

Now we show part (3). The closedness under subobjects is just Lemma \ref{lem2.1}. The closedness under quotient objects can be proven by a similar argument as in the proof of Lemma \ref{lem2.1}. The non-closedness is clear by Kato's classification theorem Theorem \ref{thm-app.3}.
\end{proof}

\begin{lem}\label{lem2.2}
Let $M=[Y\rightarrow G_{\mr{log}}],M'=[Y'\rightarrow G'_{\mr{log}}]$ be two non-degenerate log 1-motives over $S$, $(f_{-1},f_0):M\rightarrow M'$ a homomorphism of log 1-motives, and $f_{\mr{c}}:G\rightarrow G'$ the map induced by $f_0$. Let $X$ (resp. $X'$) be the character group of the torus part $T$ (resp. $T'$) of $G$ (resp. $G'$), $\mc{Q}$ (resp. $\mc{Q}'$) the sheaf $\mc{H}om_{S^{\mr{log}}_{\mr{fl}}}(X,\Gml/\Gm)^{(Y)}$ (resp. $\mc{H}om_{S^{\mr{log}}_{\mr{fl}}}(X',\Gml/\Gm)^{(Y')}$), $f_{\mr{l}}:X'\rightarrow X$ the map induced by $f_{\mr{c}}$, and $\tilde{f}_{\mr{d}}:\mc{Q}\rightarrow \mc{Q}'$ the map induced by $f_{\mr{l}}$. If $f_{\mr{c}}$ is an isogeny, then the map $\tilde{f}:G_{\mr{log}}^{(Y)}\rightarrow G'^{(Y')}_{\mr{log}}$ induced by $(f_{-1},f_0)$ is surjective with kernel $\mr{Ker}(f_{\mr{c}})$, and the map $\tilde{f}_{\mr{d}}$ is bijective.
\end{lem}
\begin{proof}
Since $f_{\mr{c}}$ is an isogeny, the map $f_{\mr{l}}$ is injective and of finite cokernel. We consider the following commutative diagram
$$\xymatrix{
0\ar[r] &G\ar[r]\ar[d]^{f_{\mr{c}}} &G_{\mr{log}}^{(Y)}\ar[r]\ar[d]^{\tilde{f}} &\mc{Q}\ar[r]\ar[d]^{\tilde{f}_{\mr{d}}} &0 \\
0\ar[r] &G'\ar[r]  &G'^{(Y')}_{\mr{log}}\ar[r]  &\mc{Q}'\ar[r] &0
}$$
with exact rows. To show the surjectivity of $\tilde{f}$, it is enough to show the surjectivity of $\tilde{f}_{\mr{d}}$. 

The induced map $T_{\mr{log}}\rightarrow T'_{\mr{log}}$ is surjective by Proposition \ref{prop1.3} (5). Furthermore we have the surjectivity of the map 
$$\mc{H}om_{S^{\mr{log}}_{\mr{fl}}}(X,\Gml/\Gm)\rightarrow \mc{H}om_{S^{\mr{log}}_{\mr{fl}}}(X',\Gml/\Gm).$$
Thus for any $\varphi'\in \mc{Q}'$, there exists some $\varphi\in \mc{H}om_{S^{\mr{log}}_{\mr{fl}}}(X,\Gml/\Gm)$ mapped to $\varphi'$, i.e. $\varphi'=\varphi\circ f_{\mr{l}}$. In order to show the surjectivity of $\tilde{f}_{\mr{d}}$, it suffices to show that $\varphi\in \mc{Q}$. Let $n$ be a positive integer killing $X/f_{\mr{l}}(X')$, and let $<,>:X\times Y\rightarrow\Gml/\Gm$ and $<,>':X'\times Y'\rightarrow \Gml/\Gm$ be the pairings associated to $M$ and $M'$ respectively. Given any $U\in (\mr{fs}/S), u\in U, x\in X_{\bar{u}}$, there exists $x'\in X'_{\bar{u}}$ such that $nx=f_{\mr{l}}(x')$. By the definition of $\mc{Q}'$, there exist   $y'_{u,x',1},y'_{u,x',2}\in Y'_{\bar{u}}$ such that 
\begin{equation}\label{eq1lem2.2}
<x',y'_{u,x',1}>'|\,\varphi'(x')\,|<x',y'_{u,x',2}>'.
\end{equation}
The map $f_{\mr{l}}$ being injective with finite cokernel, together with the non-degeneracy of $M$ and $M'$, forces $f_{-1}$ to be injective with finite cokernel. If necessary we enlarge $n$ such that it also kills the cokernel of $f_{-1}$. Then there exist $y_{u,x',1},y_{u,x',2}\in Y_{\bar{u}}$ such that $ny'_{u,x',1}=f_{-1}(y_{u,x',1})$ and $ny'_{u,x',2}=f_{-1}(y_{u,x',2})$. Raising the relation (\ref{eq1lem2.2}) to $n$-th power, we get a new relation 
\begin{equation}\label{eq2lem2.2}
<x',f_{-1}(y_{u,x',1})>'|\,\varphi'(x')^n\,|<x',f_{-1}(y_{u,x',2})>'.
\end{equation}
Since $<f_{\mr{l}}(-),->=<-,f_{-1}(-)>'$, the relation (\ref{eq2lem2.2}) can be rewritten as 
\begin{equation}\label{eq3lem2.2}
<x,y_{u,x',1}>^n|\,\varphi(x)^{n^2}\,|<x,y_{u,x',2}>^n.
\end{equation}
By \cite[18.10]{k-k-n4}, there exist $y_1, y_2\in Y_{\bar{u}}$ such that $$<x,y_1^n>|<x,y_{u,x',1}> \, \text{and}\, <x,y_{u,x',2}>|<x,y_2^n>.$$ Therefore relation (\ref{eq3lem2.2}) gives another relation
\begin{equation}\label{eq4lem2.2}
<x,y_1>^{n^2}|\,\varphi(x)^{n^2}\,|<x,y_2>^{n^2}.
\end{equation}
Removing the exponents from (\ref{eq4lem2.2}), we get $<x,y_1>|\,\varphi(x)\,|<x,y_2>$, hence $\varphi\in\mc{Q}$.

The injectivity of $\tilde{f}_{\mr{d}}$ follows from the injectivity of $$\mc{H}om_{S^{\mr{log}}_{\mr{fl}}}(X,\Gml/\Gm)\rightarrow \mc{H}om_{S^{\mr{log}}_{\mr{fl}}}(X',\Gml/\Gm).$$
The identification $\mr{Ker}(f_{\mr{c}})=\mr{Ker}(\tilde{f})$ follows from the injectivity of $\tilde{f}_{\mr{d}}$.
\end{proof}

\begin{prop}\label{prop2.2}
Let $A$ be a log abelian variety over $S$, and $F\in (\mr{fin}/S)_r$ be a subsheaf of $A$. Then:
\begin{enumerate}[(1)]
\item $F$ is an extension of objects of $(\mr{fin}/S)_c$. 
\item $F\in (\mr{fin}/S)_d$.
\item The quotient $A/F$ is also a log abelian variety over $S$.
\end{enumerate}
\end{prop}
\begin{proof}
Let $n$ be a positive integer such that $nF=0$, then we have $F\subset A[n]:=\mr{Ker}(A\xrightarrow{\times n} A)$. Since the map $G\xrightarrow{n_G}G$ is an isogeny, we get a short exact sequence $0\rightarrow G[n]\rightarrow A[n]\rightarrow (\mc{Q}/Y)[n]\rightarrow 0$ by diagram (\ref{eq1.10}), where $\mc{Q}$ denotes the sheaf $\mc{H}om(X,\Gml/\Gm)^{(Y)}$. Since the map $\mc{Q}\xrightarrow{\times n}\mc{Q}$ is an isomorphism, we get $(\mc{Q}/Y)[n]\cong Y/nY$ by diagram (\ref{eq1.12}). We also have $A[n]\in (\mr{fin}/S)_r$ by \cite[Prop. 18.1 (1) and (2)]{k-k-n4}. Now let $F'$ be the kernel of the composition $F\hookrightarrow A[n]\rightarrow (\mc{Q}/Y)[n]\cong Y/nY$, and $F''$ the image of $F$ in $Y/nY$, then we have a commutative diagram
$$\xymatrix{
0\ar[r] &F'\ar[r]\ar@{^{(}->}[d] &F\ar[r]\ar@{^{(}->}[d] &F''\ar[r]\ar@{^{(}->}[d] &0 \\
0\ar[r] &G[n]\ar[r]  &A[n]\ar[r]  &Y/nY \ar[r] &0
}$$
with exact rows and injective vertical homomorphisms. As a kernel of a homomorphism between two representable objects, $F'\in (\mr{fin}/S)_r$; as a subobject of $Y/nY$ which is a classical finite \'etale group scheme, $F''\in (\mr{fin}/S)_c$. Applying Lemma \ref{lem2.1} to the inclusion $F'\subset G[n]$, we conclude $F'\in (\mr{fin}/S)_c$, hence part (1) is proven. Part (2) follows from part (1) and Proposition \ref{prop-app.1}. 

Now we show part (3). It suffices to find a polarisable log 1-motive such that $A/F$ is isomorphic to its associated quotient. Consider the pullback $E$ of $G_{\mr{log}}^{(Y)}$ along $F\subset A$
$$\xymatrix{
0\ar[r] &Y\ar[r]\ar@{=}[d] &E\ar[r]\ar@{^{(}->}[d] &F\ar[r]\ar@{^{(}->}[d] &0 \\
0\ar[r] &Y\ar[r]  &G_{\mr{log}}^{(Y)}\ar[r]  &A \ar[r] &0,
}$$
and let $E_{\mr{tor}}$ be the torsion subsheaf of $E$, $Y':=E/E_{\mr{tor}}$. Since the sheaf $$G_{\mr{log}}^{(Y)}/G=\mc{H}om_{S^{\mr{log}}_{\mr{fl}}}(X,\Gml/\Gm)^{(Y)}$$
is torsion-free, $E_{\mr{tor}}$ maps into $G$. So we get $E_{\mr{tor}}=F'$ and $Y'/Y=F''\subset Y/nY$, and $Y'$ is \'etale locally constant. Let $G'=G/E_{\mr{tor}}=G/F'$, the inclusion $E\hookrightarrow G_{\mr{log}}^{(Y)}\hookrightarrow G_{\mr{log}}$ gives a homomorphism $Y'\rightarrow G_{\mr{log}}/F'=G'_{\mr{log}}$ by taking the quotient by $F'$. In this way, we get a log 1-motive $M':=[Y'\xrightarrow{u'} G'_{\mr{log}}]$ together with a homomorphism $(f_{-1},f_0):M:=[Y\xrightarrow{u} G_{\mr{log}}]\rightarrow M'$. By the construction of the homomorphism $(f_{-1},f_0)$, it is clear that the multiplication by $n$ map $n_M$ on $M$ factors through $(f_{-1},f_0)$. Let $(g_{-1},g_0):M'\rightarrow M$ be the homomorphism such that $n_M=(g_{-1},g_0)\circ (f_{-1},f_0)$, let $(h_{-1},h_0):M\rightarrow M^*=[X\rightarrow G^*_{\mr{log}}]$ be a polarisation of $M$, then $(g_{-1}^*\circ h_{-1}\circ g_{-1},g_0^*\circ h_0\circ g_0)$ gives rise to a polarisation of $M'$, where $(g_{-1}^*,g_0^*):M^*\rightarrow M'^*=[X'\rightarrow G'^*_{\mr{log}}]$ is the dual of $(g_{-1},g_0)$. By \cite[Thm. 3.4]{k-k-n2}, the homomorphism $(f_{-1},f_0)$ gives rise to a homomorphism $f:A\rightarrow A'$ of log abelian varieties with constant degeneration, where $A'$ is the associated log abelian variety of $M'$. By the diagram (\ref{eq1.9}) and Lemma \ref{lem2.2}, it is easy to see that $A'=A/F$.
\end{proof}

\begin{defn}\label{defn2.2}
Let $A,A'$ be two log abelian varieties over $S$. An isogeny from $A$ to $A'$ is a homomorphism $f$ from $A$ to $A'$ such that $f$ is surjective for the log flat topology and $\mr{Ker}(f)\in (\mr{fin}/S)_r$.
\end{defn}

\begin{rmk}\label{rmk2.1}
Let $f:A\rightarrow A'$ be an isogeny between two log abelian varieties over $S$. By Proposition \ref{prop2.2} (2), we have $\mr{Ker}(f)\in (\mr{fin}/S)_d$. Hence we can replace the condition $\mr{Ker}(f)\in (\mr{fin}/S)_r$ by the a priori stronger condition $\mr{Ker}(f)\in (\mr{fin}/S)_d$ in the definition of isogeny. We can also replace the condition $\mr{Ker}(f)\in (\mr{fin}/S)_r$ by the a priori weaker condition $\mr{Ker}(f)\in (\mr{fin}/S)_f$, since $\mr{Ker}(f)$ is automatically representable as the kernel of $f|_{A[n]}:A[n]\rightarrow A'[n]$ for $n$ big enough. Here we have used the fact $A[n],A'[n]\in (\mr{fin}/S)_r$, see \cite[Prop. 18.1]{k-k-n4}.
\end{rmk}

\begin{ex}\label{ex2.1}
By Proposition \ref{prop2.2}, a subsheaf $F\in (\mr{fin}/S)_r$ of a log abelian variety $A$ over $S$ gives an isogeny $A\rightarrow A/F$ of log abelian varieties.
\end{ex}

Isogenies between abelian varieties can be defined by several equivalent conditions, some of which concern the dimension. Here we show the same thing happens for log abelian varieties over $S$. 

Recall that the dimension of a log abelian variety is defined to be the dimension of its semi-abelian part, see \cite[4.4]{k-k-n2}. 

\begin{prop}\label{prop2.3}
Let $f:A\rightarrow A'$ be a homomorphism of log abelian varieties over $S$. Let $M=[Y\xrightarrow{u} G_{\mr{log}}]$ (resp. $M'=[Y'\xrightarrow{u'} G'_{\mr{log}}]$) be the log 1-motive defining $A$ (resp. $A'$), and $f_{-1}$ and $f_{\mr{c}}$ the homomorphisms induced by $f$ as in Proposition \ref{prop1.6}. Consider the following conditions:
\begin{enumerate}[(1)]
\item $f$ is an isogeny;
\item $f$ is surjective for the log flat topology and $\mr{dim}A=\mr{dim}A'$;
\item $\mr{Ker}(f)\in (\mr{fin}/S)_r$ and $\mr{dim}A=\mr{dim}A'$;
\item $f_{\mr{c}}$ is an isogeny and $f_{-1}$ is injective of finite cokernel.
\end{enumerate}
Then we have $(2)\Leftarrow (1)\Leftrightarrow (3)\Leftrightarrow (4)$.
\end{prop}
\begin{proof}
If $f$ is an isogeny, we have $\mr{Ker}(f)\in (\mr{fin}/S)_r$ and $A'=A/\mr{Ker}(f)$. By the construction of $A/\mr{Ker}(f)$ as a log abelian variety in the proof of part (3) of Proposition \ref{prop2.2}, we have that $f_{\mr{c}}$ is an isogeny of semi-abelian varieties, hence $\mr{dim}G=\mr{dim}G'$. This shows that (1) implies both (2) and (3).

Now we show that (3) implies (1). The condition $\mr{Ker}(f)\in (\mr{fin}/S)_r$ implies that $\mr{Ker}(f_{\mr{c}})$ is a finite group scheme. Since $\mr{dim}G=\mr{dim}A=\mr{dim}A'=\mr{dim}G'$, $f_{\mr{c}}$ is an isogeny. Let $\tilde{f}:G_{\mr{log}}^{(Y)}\rightarrow G'^{(Y')}_{\mr{log}}$ be the homomorphism induced by $f$ as in Proposition \ref{prop1.6}. By Lemma \ref{lem2.2}, $\tilde{f}$ is surjective. Then the surjectivity of $f$ follows from the surjectivity of $\tilde{f}$. Hence $f$ is an isogeny.

At last, we show the equivalence between (3) and (4). Assuming (3), we must have that $f_{\mr{c}}$ is an isogeny. Let $\tilde{f}_{\mr{d}}:\mc{Q}\rightarrow\mc{Q}'$ be the homomorphism induced by $f$ as in Proposition \ref{prop1.6}. Then $\tilde{f}_{\mr{d}}$ is bijective by Lemma \ref{lem2.2}. Applying snake lemma to diagrams (\ref{eq1.10}) and (\ref{eq1.12}), we get $\mr{Ker}(f_{-1})=0$ and a short exact sequence $0\rightarrow \mr{Ker}(f_{\mr{c}})\rightarrow \mr{Ker}(f)\rightarrow \mr{Coker}(f_{-1})\rightarrow 0$. Hence (3) implies (4). Conversely, assuming (4), we have that $\tilde{f}_{\mr{d}}$ is bijective by Lemma \ref{lem2.2}. Again applying snake lemma to diagrams (\ref{eq1.10}) and (\ref{eq1.12}), we have a short exact sequence $0\rightarrow \mr{Ker}(f_{\mr{c}})\rightarrow \mr{Ker}(f)\rightarrow \mr{Coker}(f_{-1})\rightarrow 0$. Hence $\mr{Ker}(f)\in (\mr{fin}/S)_r$ by part (2) of Proposition \ref{prop2.1}, so we get (3).
\end{proof}

\begin{rmk}\label{rmk2.2}
One might wonder if (2) implies (1) in Proposition \ref{prop2.3}. Note that the corresponding statement for abelian varieties holds. It is easy to see that the implication follows from the surjectivity of $f_{\mr{c}}$. Unfortunately, it is not clear to the author how to deduce the surjectivity of $f_{\mr{c}}$ from the surjectivity of $f$.
\end{rmk}

\begin{ex}\label{ex2.2}
Let $A$ be a log abelian variety over $S$. Let $M=[Y\rightarrow G_{\mr{log}}]$ be the log 1-motive defining $A$, $M^*=[X\rightarrow G^*_{\mr{log}}]$ the dual of $M$ and $(\lambda_{-1},\lambda_0):M\rightarrow M^*$ a polarisation, see \cite[Def. 2.8]{k-k-n2} for the definition of polarisation. Then the map $\lambda:A\rightarrow A^*$ induced by $(\lambda_{-1},\lambda_0)$ is an isogeny. One calls $\lambda$ a polarisation of the log abelian variety $A$.
\end{ex}

\begin{prop}\label{prop2.4}
Let $A$ be a log abelian variety over $S$, let $g$ be the dimension of $A$, and let $n$ be a positive integer.
\begin{enumerate}[(1)]
\item The multiplication-by-$n$ map $n_A:A\rightarrow A$ is an isogeny.
\item The rank of $A[n]:=\mr{Ker}(n_A)$ is $n^{2g}$.
\item $A[n]\in (\mr{fin}/S)_d$.
\item Let $(n_A)^*$ be the dual of the map $n_A$, then $(n_A)^*=n_{A^*}$.
\item If $n$ is coprime to the characteristic of $k$, then Kummer \'etale locally on $S$, $A[n]$ is isomorphic to $(\Z/n\Z)^{2g}$.
\end{enumerate}
\end{prop}
\begin{proof}
By \cite[Prop. 18.1]{k-k-n4}, we have $A[n]\in (\mr{fin}/S)_r$, hence part (1) is a corollary of Proposition \ref{prop2.3}. For part (2) and part (5), we refer to \cite[Prop. 18.1]{k-k-n4}. Part (3) follows from part (1) and Remark \ref{rmk2.1}. We are left with part (4). Let $M=[Y\rightarrow G_{\mr{log}}]$ be the log 1-motive defining $A$, then $n_A$ is the map induced by the map $n_M$. Since the dual of $n_M$ is the map $n_{M^*}$, where $M^*$ denote the dual of $M$, the dual of $n_A$ is nothing but $n_{A^*}$.
\end{proof}

\subsection{The dual short exact sequence}
Recall that for an isogeny $f:A\rightarrow A'$ between two abelian varieties over a field, we have that the dual $f^*$ of $f$ is an isogeny with kernel $(\mr{Ker}(f))^*$. In this subsection we show that the same thing holds for log abelian varieties over $S$.

Let $f:A\rightarrow A'$ be an isogeny between two log abelian varieties over $S$, and let $F$ be the kernel of $f$, then we get a short exact sequence $0\rightarrow F\rightarrow A\xrightarrow{f} A'\rightarrow 0$. Applying the functor $\mc{H}om_{S^{\mr{log}}_{\mr{fl}}}(-,\Gml)$ to this short exact sequence, we get a long exact sequence
\begin{equation}\label{eq2.1}
\begin{split}
\rightarrow &\mc{H}om_{S^{\mr{log}}_{\mr{fl}}}(A,\Gml) \rightarrow \mc{H}om_{S^{\mr{log}}_{\mr{fl}}}(F,\Gml)  \rightarrow \mc{E}xt_{S^{\mr{log}}_{\mr{fl}}}(A',\Gml)  \\
\rightarrow &\mc{E}xt_{S^{\mr{log}}_{\mr{fl}}}(A,\Gml).
\end{split}
\end{equation}
By Theorem \ref{thm1.2} (4), $\mc{H}om_{S^{\mr{log}}_{\mr{fl}}}(A,\Gml)=0$. By Theorem \ref{thm1.2} (3), the map $\mc{E}xt_{S^{\mr{log}}_{\mr{fl}}}(A',\Gml)\rightarrow \mc{E}xt_{S^{\mr{log}}_{\mr{fl}}}(A,\Gml)$ is just the map $G'^*_{\mr{log}}/X'\rightarrow G^*_{\mr{log}}/X$. The torsion-free nature of $\Gml/\Gm$ implies $\mc{H}om_{S^{\mr{log}}_{\mr{fl}}}(F,\Gml/\Gm)=0$, hence we have $\mc{H}om_{S^{\mr{log}}_{\mr{fl}}}(F,\Gml)=F^*$, where $F^*=\mc{H}om_{S^{\mr{log}}_{\mr{fl}}}(F,\Gm)$ is the Cartier dual of $F$ (see Definition \ref{defn-app.1}). We have $A^*=G^{*(X)}_{\mr{log}}/X$ and $A'^*=G'^{*(X')}_{\mr{log}}/X'$. And the map $X'\rightarrow G'^*_{\mr{log}}$ having its image in $G'^{*(X')}_{\mr{log}}$ gives a short exact sequence
$$0\rightarrow A'^*\rightarrow G'^*_{\mr{log}}/X'\rightarrow \mc{R}\rightarrow 0,$$
where $\mc{R}$ denotes the quotient sheaf
$$\mc{H}om_{S^{\mr{log}}_{\mr{fl}}}(Y',\Gml/\Gm)/\mc{H}om_{S^{\mr{log}}_{\mr{fl}}}(Y',\Gml/\Gm)^{(X')}.$$
Putting all these ingredients together, we get a commutative diagram
\begin{equation}\label{eq2.2}
\xymatrix{
&&0\ar[d]   \\
&&A'^*\ar[r]^{f^*}\ar@{^(->}[d] &A^*\ar@{^(->}[d]  \\
0\ar[r]&F^*\ar[r]&G'^*_{\mr{log}}/X'\ar[r]\ar[d] &G^*_{\mr{log}}/X  \\
&&\mc{R}\ar[d]  \\
&&0
}
\end{equation}
with exact rows and columns.

\begin{lem}\label{lem2.3}
The sheaf $\mc{R}$ is torsion-free.
\end{lem}
\begin{proof}
We consider the following short exact sequence 
$$0\rightarrow \mc{Q}'^*\rightarrow \mc{H}om_{S^{\mr{log}}_{\mr{fl}}}(Y',\Gml/\Gm)\rightarrow \mc{R}\rightarrow 0,$$
where $\mc{Q}'^*$ denotes $\mc{H}om_{S^{\mr{log}}_{\mr{fl}}}(Y',\Gml/\Gm)^{(X')}$. To show $\mc{R}$ is torsion-free, it is enough to show that any section $\varphi\in \mc{H}om_{S^{\mr{log}}_{\mr{fl}}}(Y',\Gml/\Gm)$, satisfying $\varphi^n\in\mc{Q}'^*$, actually lies in $\mc{Q}'^*$. For any $U\in (\mr{fs}/S),u\in U, y'\in Y'_{\bar{u}}$, there exist $x'_{u,y',1},x'_{u,y',2}\in X'_{\bar{u}}$ such that $<x'_{u,y',1},y'>'|\, \varphi^n(y') \,|<x'_{u,y',2},y'>'$. By \cite[18.10]{k-k-n4}, there exist $x_1,x_2\in X'_{\bar{u}}$ such that $$<x_1,y'>'^{n}|<x'_{u,y',1},y'>',\quad <x'_{u,y',2},y'>'|<x_2,y'>'^n.$$
Then we get a relation $<x_1,y'>'^n|\,  \varphi^n(y') \,|<x_2,y'>'^n$. Removing the exponents, we further get $<x_1,y'>'|\,  \varphi(y') \,|<x_2,y'>'$, which shows that $\varphi\in \mc{Q}'^*$.
\end{proof}

\begin{thm}\label{thm2.1}
We have a canonical short exact sequence $$0\rightarrow F^*\rightarrow  A'^*\xrightarrow{f^*} A^*\rightarrow 0,$$
in other words, $f^*$ is an isogeny with kernel the Cartier dual of $F$.
\end{thm}
\begin{proof}
Since $F\in (\mr{fin}/S)_d$ by Proposition \ref{prop2.2} (2), $F^*\in (\mr{fin}/S)_d$. The sheaf $\mc{R}$ is torsion-free by Lemma \ref{lem2.3}, hence we have $\mr{Hom}_{S^{\mr{log}}_{\mr{fl}}}(F^*,\mc{R})=0$. It follows that the map $F^*\hookrightarrow G'^*_{\mr{log}}/X'$ in the diagram (\ref{eq2.2}) factors through $A'^*$. Furthermore, $F^*$ is actually the kernel of $f^*$.  Since $\mr{dim}A'^*=\mr{dim}A^*$, $f^*$ is an isogeny by Proposition \ref{prop2.3}.
\end{proof}

\subsection{The Poincar\'e complete reducibility theorem}
The Poincar\'e complete reducibility theorem for abelian varieties plays a very important role in the theory of abelian varieties. In this subsection, we formulate a Poincar\'e complete reducibility theorem for log abelian varieties admitting a polarisation over $S$.

\begin{lem}\label{lem2.4}
Let $M=[Y\rightarrow G_{\mr{log}}]$ be a log 1-motive over $S$ with a polarisation $(\lambda_{-1},\lambda_0):M=[Y\rightarrow G_{\mr{log}}]\rightarrow [X\rightarrow G^*_{\mr{log}}]=M^*$, and let $A$ be the log abelian variety associated to $M$. Let $M_1=[Y_1\rightarrow G_{1\mr{log}}]$ be another log 1-motive with $\mr{rank}_{\Z}Y_1=\mr{rank}_{\Z}X_1$, where $X_1$ is the character group of the torus part of $G_1$. Let $(i_{-1},i_0):M_1\rightarrow M$ be a homomorphism of log 1-motives and $i_{\mr{c}}:G_1\rightarrow G$ the homomorphism corresponding to $i_0$, and let $\gamma_{-1}:=i^*_{-1}\circ\lambda_{-1}\circ i_{-1}$, $\gamma_0:=i^*_0\circ\lambda_0\circ i_0$. Suppose that $i_{-1}$ is injective and $i_{\mr{c}}$ has finite kernel, then we have the following.
\begin{enumerate}[(1)]
\item The map $(\gamma_{-1},\gamma_0):M_1=[Y_1\rightarrow G_{1\mr{log}}]\rightarrow [X_1\rightarrow G^*_{1\mr{log}}]=M_1^*$ is a polarisation.
\item Let $A_1$ be the log abelian variety associated to $M_1$ and $i:A_1\rightarrow A$ the map induced by $(i_{-1},i_0)$, then we have $\mr{Ker}(i)\in (\mr{fin}/S)_r$.
\end{enumerate}
\end{lem}
\begin{proof}
To prove part (1), we need to verify the conditions (a), (b), (c) and (d) of \cite[Def. 2.8]{k-k-n2}. We have the following commutative diagram
$$\xymatrix{
Y_1\ar[r]^{i_{-1}}\ar[d] &Y\ar[r]^{\lambda_{-1}}\ar[d] &X\ar[r]^{i^*_{-1}}\ar[d] & X_1\ar[d]   \\
G_{1\mr{log}} \ar[r]^{i_0} &G_{\mr{log}}\ar[r]^{\lambda_0} &G^*_{\mr{log}}\ar[r]^{i^*_0} &G^*_{1\mr{log}}.
}$$
We also have the following commutative diagram
$$\xymatrix{
0\ar[r] &T_1\ar[r]\ar[d]^{i_{\mr{t}}} &G_1\ar[r]\ar[d]^{i_{\mr{c}}} &B_1\ar[r]\ar[d]^{i_{\mr{ab}}} &0 \\
0\ar[r] &T\ar[r]\ar[d]^{\lambda_{\mr{t}}}  &G\ar[r]\ar[d]^{\lambda_{\mr{c}}}  &B\ar[r]\ar[d]^{\lambda_{\mr{ab}}} &0  \\
0\ar[r] &T^*\ar[r]\ar[d]^{i^*_{\mr{t}}} &G^*\ar[r]\ar[d]^{i^*_{\mr{c}}} &B^*\ar[r]\ar[d]^{i^*_{\mr{ab}}} &0 \\
0\ar[r] &T^*_1\ar[r]  &G^*_1\ar[r]  &B^*_1\ar[r] &0
}$$
with exact rows, where the rows are the torus and abelian variety decomposition exact sequences of semi-abelian varieties. By the construction of the duality theory of log 1-motives, we have that $i^*_{-1}$ (resp. $i_{\mr{t}}^*$) is induced by $i_{\mr{t}}$ (resp. $i_{-1}$). Then condition (d) follows. Let $<,>:X\times Y\rightarrow \Gml/\Gm$ (resp. $<,>_1:X_1\times Y_1\rightarrow \Gml/\Gm$) be the pairing associated to $M$ (resp. $M_1$). For $y\in Y_{1\bar{s}}\backslash \{0\}$, where $s$ denotes the only point of $S$, we have $i_{-1}(y)\neq 0$ by the injectivity of $i_{-1}$. Hence we have
$$<\gamma_{-1}(y),y>_{1\bar{s}}=<\lambda_{-1}\circ i_{-1}(y),i_{-1}(y)>_{\bar{s}}\in (M_{S,\bar{s}}/\mc{O}^{\times}_{S,\bar{s}})\backslash \{1\}$$
which gives condition (c). For condition (b), it suffices to show the injectivity of $\gamma_{-1}$ because of $\mr{rank}_{\Z} Y_1=\mr{rank}_{\Z} X_1$. But this already follows from condition (c). At last we show condition (a). Since $i_{\mr{c}}$ has finite kernel, $i_{\mr{ab}}$ must have finite kernel by diagram chasing. Hence $i_{\mr{ab}}$ is a finite morphism. We want to show $\gamma_{\mr{ab}}=i_{\mr{ab}}^*\circ\lambda_{\mr{ab}}\circ i_{\mr{ab}}$ is a polarisation of $B_1$. Without loss of generality, we may assume $\lambda_{\mr{ab}}=\varphi_{\mc{L}}$ for an ample line bundle $\mc{L}$ on $B$, where $\varphi_{\mc{L}}$ is defined by $\varphi_{\mc{L}}(b):=t_b^*\mc{L}\otimes \mc{L}^{-1}$ for $b\in B$. Then we have $\gamma_{\mr{ab}}=i_{\mr{ab}}^*\circ \varphi_{\mc{L}}\circ i_{\mr{ab}}=\varphi_{i_{\mr{ab}}^*\mc{L}}$. Clearly $i_{\mr{ab}}^*\mc{L}$ is ample, so $\gamma_{\mr{ab}}$ is a polarisation on $B_1$. This finishes the verification of condition (a). Hence $(\gamma_{-1},\gamma_0)$ is a polarisation of $M_1$. 

Now we prove part (2). By Proposition \ref{prop1.6}, the homomorphism $i$ induces the following two commutative diagrams
$$\xymatrix{
0\ar[r] &G_1\ar[r]\ar[d]^{i_{\mr{c}}} &A_1\ar[r]\ar[d]^i &\mc{Q}_1/Y_1\ar[r]\ar[d]^{i_{\mr{d}}} &0 \\
0\ar[r] &G\ar[r]  &A\ar[r]  &\mc{Q}/Y\ar[r] &0 }$$
and
$$\xymatrix{0\ar[r] &Y_1\ar[r]\ar[d]^{i_{-1}} &\mc{Q}_1\ar[r]\ar[d]^{\tilde{i}_{\mr{d}}}  &\mc{Q}_1/Y_1\ar[r]\ar[d]^{i_{\mr{d}}} &0 \\
0\ar[r] &Y\ar[r]  &\mc{Q}\ar[r]  &\mc{Q}/Y\ar[r] &0
}$$
with exact rows. Since $i_{\mr{c}}$ has finite kernel, $i_{\mr{t}}$ has finite kernel too. Hence $\tilde{i}_{\mr{d}}$ is injective by part (3) Proposition \ref{prop1.3}. Applying snake lemma to the above two diagrams, we get exact sequences $0\rightarrow \mr{Ker}(i_{\mr{c}})\rightarrow\mr{Ker}(i)\rightarrow  \mr{Ker}(i_{\mr{d}})\rightarrow \mr{Coker}(i_{\mr{c}})$ and $0\rightarrow \mr{Ker}(i_{\mr{d}})\xrightarrow{\alpha} Y/Y_1$.
\begin{claim}
The map $\alpha$ maps $\mr{Ker}(i_{\mr{d}})$ onto the torsion part $(Y/Y_1)_{\mr{tor}}$ of $Y/Y_1$.  
\end{claim}
\begin{proof}[Proof of CLAIM]
First we show $(Y/Y_1)_{\mr{tor}}\subset\mr{Im}(\alpha)$. Let $y\in Y$ be such that the quotient class $\bar{y}$ represented by $y$ lies in $(Y/Y_1)_{\mr{tor}}$. We have $ny=i_{-1}(y_1)$ for some $n\in\N$ and $y_1\in Y_1$. The map $\mc{Q}_1\xrightarrow{\times n}\mc{Q}_1$ is bijective by Lemma \ref{lem2.2}, hence there exists $\varphi_1\in\mc{Q}_1$ such that $\varphi_1^n=<-,y_1>_1$. It follows then $\tilde{i}_{\mr{d}}(\varphi_1)^n=<-,y>^n$. Since the map $\mc{Q}\xrightarrow{\times n}\mc{Q}$ is bijective by Lemma \ref{lem2.2}, we get $\tilde{i}_{\mr{d}}(\varphi_1)=<-,y>$. Whence $\bar{y}=\alpha(\bar{\varphi}_1)$, where $\bar{\varphi}_1$ denotes the quotient class of $\varphi_1$ in $\mc{Q}_1/Y_1$.

Now we show the converse inclusion. It suffices to show, for $\varphi_1\in\mc{Q}_1$ with $\tilde{i}_{\mr{d}}(\varphi_1)=<-,y>$ for some $y\in Y$, $\bar{y}\in (Y/Y_1)_{\mr{tor}}$. Since $\gamma_{-1}=i^*_{-1}\circ\lambda_{-1}\circ i_{-1}$ is of finite cokernel by part (1), there exists some positive integer $m$ such that $mi^*_{-1}\circ\lambda_{-1}(y)=i^*_{-1}\circ\lambda_{-1}\circ i_{-1}(y_1)$ for some $y_1\in Y_1$. Let $w:=my-i_{-1}(y_1)$, and we have $i_{-1}^*\circ\lambda_{-1}(w)=0$. Hence
\begin{align*}
<\lambda_{-1}(w),w>&=<\lambda_{-1}(w),my>-<\lambda_{-1}(w),i_{-1}(y_1)> \\
&=\tilde{i}_{\mr{d}}(\varphi_1)(\lambda_{-1}(w))^m-<i^*_{-1}\circ\lambda_{-1}(w),y_1>_1\\
&=\varphi_1(i_{-1}^*\circ\lambda_{-1}(w))^m-<0,y_1>_1  \\
&=0.
\end{align*}
Since $(\lambda_{-1},\lambda_0)$ is a polarisation, we must have $w=0$. This shows $\bar{y}\in (Y/Y_1)_{\mr{tor}}$.
\end{proof}

By the above claim, we get an exact sequence $$0\rightarrow \mr{Ker}(i_{\mr{c}})\rightarrow\mr{Ker}(i)\rightarrow  (Y/Y_1)_{\mr{tor}}\rightarrow \mr{Coker}(i_{\mr{c}}).$$
Let $F$ be the kernel of $(Y/Y_1)_{\mr{tor}}\rightarrow \mr{Coker}(i_{\mr{c}})$. We have $F\in (\mr{fin}/S)_c$. Then the short exact sequence $0\rightarrow \mr{Ker}(i_{\mr{c}})\rightarrow\mr{Ker}(i)\rightarrow F\rightarrow 0$ forces $\mr{Ker}(i)\in (\mr{fin}/S)_r$ by Proposition \ref{prop-app.1}. 
\end{proof}

\begin{prop}\label{prop2.5}
Let $f:A\rightarrow A'$ be a homomorphism of log abelian varieties over $S$. Then there exists a log abelian subvariety $j:A_1\hookrightarrow A$ such that $f|_{A_1}=0$, and $A_1$ possesses the following universal property: for any homomorphism $g:A_2\rightarrow A$ of log abelian varieties over $S$ such that $f\circ g=0$, $g$ factors through $A_1$ uniquely. In other words, $A_1$ is the kernel of $f$ in the category of log abelian varieties over $S$.
\end{prop}
\begin{proof}
Let $M=[Y\rightarrow G_{\mr{log}}]$ (resp. $M'=[Y'\rightarrow G'_{\mr{log}}]$) be the log 1-motive defining $A$ (resp. $A'$), and let $(f_{-1},f_0):M\rightarrow M'$ be the homomorphism defining $f$.

We first construct the log 1-motive defining $A_1$. The homomorphism $f$ induces a homomorphism $f_{\mr{c}}:G\rightarrow G'$. Let $G_1$ be the reduced neutral component of $\mr{Ker}(f_{\mr{c}})$, then $G_1$ is a semi-abelian variety by \cite[Rem. 5.4.7. (iii)]{bri1}. Let $j_{\mr{c}}$ be the inclusion $G_1\subset G$, and let $j_0:G_{1\mr{log}}\rightarrow G_{\mr{log}}$ be the map induced by $j_{\mr{c}}$. We consider the following commutative diagram
$$\xymatrix{
0\ar[r] &G_1\ar[r]\ar[d]^{j_{\mr{c}}} &G_{1\mr{log}}\ar[r]\ar[d]^{j_0} &\mc{H}om_{S^{\mr{log}}_{\mr{fl}}}(X_1,\bar{\mathbb{G}}_{\mr{m,log}})\ar[r]\ar[d]^{\bar{j_0}} &0 \\
0\ar[r] &G\ar[r]\ar[d]^{f_{\mr{c}}}  &G_{\mr{log}}\ar[r]\ar[d]^{f_0}  &\mc{H}om_{S^{\mr{log}}_{\mr{fl}}}(X,\bar{\mathbb{G}}_{\mr{m,log}}) \ar[r]\ar[d]^{\bar{f_0}} &0   \\
0\ar[r] &G'\ar[r]  &G'_{\mr{log}}\ar[r]  &\mc{H}om_{S^{\mr{log}}_{\mr{fl}}}(X',\bar{\mathbb{G}}_{\mr{m,log}}) \ar[r] &0  
}$$
with exact rows, where $X_1$ (resp. $X$, resp. $X'$) is the character group of the torus part $T_1$ (resp. $T$, resp. $T'$) of $G_1$ (resp. $G$, resp. $G'$).  By part (3) of Proposition \ref{prop1.3}, we have $$\mr{Ker}(\bar{f_0})=\mc{H}om_{S^{\mr{log}}_{\mr{fl}}}(\mr{Coker}(f_{\mr{l}}),\bar{\mathbb{G}}_{\mr{m,log}}),$$
where $f_{\mr{l}}:X'\rightarrow X$ is the map induced by the torus part $f_{\mr{t}}:T\rightarrow T'$ of $f_{\mr{c}}$. 
The above diagram gives rise to another commutative diagram
$$\xymatrix{
0\ar[r] &G_1\ar[r]\ar@{^(->}[d]^{j_{\mr{c}}} &G_{1\mr{log}}\ar[r]\ar[d]^{j_0} &\mc{H}om_{S^{\mr{log}}_{\mr{fl}}}(X_1,\bar{\mathbb{G}}_{\mr{m,log}})\ar[r]\ar[d]^{\bar{j_0}} &0 \\
0\ar[r] &\mr{Ker}(f_{\mr{c}})\ar[r]  &\mr{Ker}(f_0)\ar[r]^-{\alpha}  &\mc{H}om_{S^{\mr{log}}_{\mr{fl}}}(\mr{Coker}(f_{\mr{l}}),\Gmlb)    
}$$
with exact rows. Since the map $\mr{Coker}(f_{\mr{l}})\rightarrow X_1$ induced by $G_1\rightarrow G\rightarrow G'$ is an isomorphism up to torsion, the map $\bar{j_0}$ in the above diagram is an isomorphism. Then we have that $\alpha$ is surjective. By the snake lemma, we get that $G_{1\mr{log}}$ is canonically embedded into $\mr{Ker}(f_0)$ with finite cokernel $\mr{Ker}(f_{\mr{c}})/G_1$.

Now let $Y_1$ be the pullback of $G_{1\mr{log}}$ along $\mr{Ker}({f_{-1}})\hookrightarrow \mr{Ker}(f_0)$, then $\mr{rank}_{\Z}Y_1=\mr{rank}_{\Z}\mr{Ker}({f_{-1}})$. Let $j_{-1}$ be the canonical inclusion $Y_1\subset Y$, we get a log 1-motive $M_1:=[Y_1\rightarrow G_{1\mr{log}}]$ together with a canonical map $(j_{-1},j_0):M_1\rightarrow M$. In order to apply Lemma \ref{lem2.4}, we need to show $\mr{rank}_{\Z}Y_1=\mr{rank}_{\Z}X_1$. Without loss of generality, we may assume that both $M$ and $M'$ admit a polarisation, in particular we have homomorphisms $h:Y\rightarrow X$ and $h':Y'\rightarrow X'$ which are both injective with finite cokernel. Consider the following diagram
$$\renewcommand{\arraystretch}{1.3}
\begin{array}[c]{ccccc}
\mr{Coker}(f^*_{-1})&\times &\mr{Ker}(f_{-1})    \\
\uparrow\scriptstyle{d} &&\downarrow\scriptstyle{a}  \\
X&\times&Y&\xrightarrow{<,>}&\Gml/\Gm \\
\uparrow\scriptstyle{f^*_{-1}}&&\downarrow\scriptstyle{f_{-1}}&&\rotatebox{90}{$=$}\\
X'&\times&Y'&\xrightarrow{<,>'}&\Gml/\Gm  \\
\uparrow\scriptstyle{c} &&\downarrow\scriptstyle{b}   \\
\mr{Ker}(f^*_{-1}) &\times &\mr{Coker}(f_{-1})
\end{array}$$
in which the parings $<,>$ and $<,>'$ are compatible with the maps $f^*_{-1}$ and $f_{-1}$, we have the following relations
\begin{enumerate}[(1)]
\item the composition $\mr{Ker}(f_{-1})\xrightarrow{a} Y\xrightarrow{h} X\xrightarrow{d}\mr{Coker}(f^*_{-1})$ is injective, whence $\mr{rank}_{\Z}\mr{Ker}(f_{-1})\leq \mr{rank}_{\Z}\mr{Coker}(f^*_{-1})$;
\item the composition $\mr{Ker}(f^*_{-1})\otimes\Q\xrightarrow{c_{\Q}} X'\otimes\Q\xrightarrow{h'^{-1}_{\Q}} Y'\otimes\Q\xrightarrow{b_{\Q}}\mr{Coker}(f_{-1})\otimes\Q$ is injective, whence $\mr{rank}_{\Z}\mr{Ker}(f^*_{-1})\leq \mr{rank}_{\Z}\mr{Coker}(f_{-1})$;
\item $\mr{rank}_{\Z}X=\mr{rank}_{\Z}Y$, $\mr{rank}_{\Z}X'=\mr{rank}_{\Z}Y'$;
\item $\mr{rank}_{\Z}\mr{Ker}(f_{-1})-\mr{rank}_{\Z}\mr{Coker}(f_{-1})=\mr{rank}_{\Z}Y-\mr{rank}_{\Z}Y'$;
\item $\mr{rank}_{\Z}\mr{Coker}(f^*_{-1})-\mr{rank}_{\Z}\mr{Ker}(f^*_{-1})=\mr{rank}_{\Z}X-\mr{rank}_{\Z}X'$.
\end{enumerate}
The relations (3), (4) and (5) are trivial. For $y\in \mr{Ker}(f_{-1})$ such that $d\circ h(y)=0$, we have $h(y)=f^*_{-1}(x')$ for some $x'\in X'$, hence $$0=<x',f_{-1}(y)>'=<f^*_{-1}(x'),y>=<h(y),y>.$$ This implies $y=0$, hence relation (1). Relation (2) can be shown by a similar argument. Now these five relations together force $\mr{rank}_{\Z}\mr{Ker}(f_{-1})=\mr{rank}_{\Z}\mr{Coker}(f^*_{-1})$, so we get $\mr{rank}_{\Z}Y_1=\mr{rank}_{\Z}X_1$.

Applying Lemma \ref{lem2.4} to $M_1$ (if necessary we take base change to $\bar{k}$ in order to get a polarisation on $M$), we have that $M_1$ defines a log abelian variety $A_1$ and $(j_{-1},j_0)$ gives a homomorphism $j:A_1\rightarrow A$. We leave the proof of the injectivity of $j$ to Lemma \ref{lem2.5}.

Now we are left with checking the universal property. Let 
$$(g_{-1},g_0):M_2=[Y_2\rightarrow G_{2\mr{log}}]\rightarrow M$$
be the homomorphism defining $g$. By the equivalence of categories from Theorem \ref{thm1.5}, we have that $f\circ g=0$ implies $f_{-1}\circ g_{-1}=f_{0}\circ g_{0}=f_{\mr{c}}\circ g_{\mr{c}}=0$. Hence the map $g_{\mr{c}}$ factors through the map $j_{\mr{c}}$ uniquely, further the map $g_0$ factors through $j_0$ uniquely. The equality $f_{-1}\circ g_{-1}=0$ implies that $g_{-1}(Y_2)\subset \mr{Ker}(f_{-1})$. Since $Y_1$ is defined as the pullback of $G_{1\mr{log}}$ along $\mr{Ker}({f_{-1}})\hookrightarrow \mr{Ker}(f_0)$,  the homomorphism $Y_2\rightarrow \mr{Ker}(f_{-1})$ factors through $Y_1$ uniquely. It follows that $(g_{-1},g_0)$ factors through $(j_{-1},j_0)$ uniquely, and $g$ factors through $j$ uniquely.
\end{proof}  

\begin{lem}\label{lem2.5}
The homomorphism $j:A_1\rightarrow A$ in the proof of Proposition \ref{prop2.5} is injective.
\end{lem}
\begin{proof}
Let the notation be as in the proof of Proposition \ref{prop2.5}. We have the following commutative diagram 
$$\xymatrix{
Y_1\ar@{^{(}->}[r]\ar@{^{(}->}[d] &\mr{Ker}(f_{-1})\ar@{^{(}->}[r]\ar@{^{(}->}[d] &Y\ar[r]^{f_{-1}}\ar@{^{(}->}[d] &Y'\ar@{^{(}->}[d]  \\
G_{1\mr{log}}\ar@{^{(}->}[r] &\mr{Ker}(f_0)\ar@{^{(}->}[r] &G_{\mr{log}}\ar[r]^{f_0} &G'_{1\mr{log}} .
}$$
Since the left square is a pullback diagram, we have that the canonical map $G_{1\mr{log}}/Y_1\rightarrow \mr{Ker}(f_0)/\mr{Ker}(f_{-1})$ is injective. By diagram chasing, the canonical map $\mr{Ker}(f_0)/\mr{Ker}(f_{-1})\rightarrow G_{\mr{log}}/Y$ is also injective. Hence the canonical map $G_{1\mr{log}}/Y_1\rightarrow G_{\mr{log}}/Y$ is injective. Then the injectivity of $j$ follows.
\end{proof}

\begin{lem}\label{lem2.6}
Let $i:A_1\hookrightarrow A$ be an inclusion of log abelian varieties over $S$. Let $(i_{-1},i_0):M_1=[Y_1\rightarrow G_{1\mr{log}}]\rightarrow [Y\rightarrow G_{\mr{log}}]=M$ be the homomorphism of log 1-motives defining $i$, and let $i_{\mr{c}}:G_1\rightarrow G$ be the homomorphism of semi-abelian varieties induced by $i_0$. Then $i_{-1}$, $i_0$ and $i_{\mr{c}}$ are all injective.
\end{lem}
\begin{proof}
The injectivity of $i_{\mr{c}}$ follows from the injectivity of $i$ by diagram (\ref{eq1.10}). The injectivity of $i_0$ follows from the injectivity of $i_{\mr{c}}$ by part (4) of Proposition \ref{prop1.3}. The injectivity of $i_{-1}$ follows from that of $i_0$ by diagram (\ref{eq1.9}).
\end{proof}

\begin{thm}[Poincar\'e complete reducibility theorem]\label{thm2.2}
Let $A$ be a log abelian variety over $S$ with a polarisation $\lambda:A\rightarrow A^*$, and $A_1$ a log abelian subvariety of $A$. Then there is another log abelian subvariety $A_2$ such that $A_1\times A_2$ is isogenous to $A$. 
\end{thm}
\begin{proof}
Let $M=[Y\rightarrow G_{\mr{log}}]$ and $M_1=[Y_1\rightarrow G_{1\mr{log}}]$ be the log 1-motives defining $A$ and $A_1$ respectively. Let $i$ be the inclusion $A_1\subset A$, let $i^*$ be the dual of $i$, and let $f=i^*\circ\lambda$. Let $(f_{-1},f_0):M\rightarrow M_1^*$ be the homomorphism defining $f$.  By Proposition \ref{prop1.6}, we have a commutative diagram
$$\xymatrix{
0\ar[r] &G\ar[r]\ar[d]_{\lambda_{\mr{c}}}\ar@{.>}@/^2pc/[dd]^<<<<<<<{f_{\mr{c}}}  &A\ar[r]\ar[d]_{\lambda}\ar@{.>}@/^2pc/[dd]^<<<<<<<f &\mc{Q}/Y\ar[r]\ar[d]_{\lambda_{\mr{d}}}\ar@{.>}@/^2pc/[dd]^<<<<<<<{f_{\mr{d}}} &0  \\
0\ar[r] &G^*\ar[r]\ar[d]_{i^*_{\mr{c}}} &A^*\ar[r]\ar[d]_{i^*} &\mc{Q}^*/X\ar[r]\ar[d]_{i^*_{\mr{d}}} &0 \\
0\ar[r] &G_1^*\ar[r] &A_1^*\ar[r]  &\mc{Q}_1^*/X_1\ar[r] &0
}$$
with exact rows. 

Firstly we study the homomorphism $f_{\mr{c}}$ via $\lambda_{\mr{c}}$ and $i_{\mr{c}}^*$. By Example \ref{ex2.2}, $\lambda$ is an isogeny, hence $\mr{Ker}(\lambda)\in (\mr{fin}/S)_r$ and $\mr{dim}G=\mr{dim}G^*$ by Proposition \ref{prop2.3}. So $\lambda_{\mr{c}}$ is an isogeny. The construction of $i_{\mr{c}}^*$ gives a commutative diagram 
$$\xymatrix{
0\ar[r] &T^*\ar[r]\ar[d]^{i_{\mr{t}}^*} &G^*\ar[r]\ar[d]^{i^*_{\mr{c}}} &B^*\ar[r]\ar[d]^{i_{\mr{ab}}^*} &0 \\
0\ar[r] &T^*_1\ar[r] &G_1^*\ar[r]  &B^*_1\ar[r] &0
}$$
with exact rows, where the map $T^*\rightarrow T^*_1$ is induced by $i_{-1}:Y_1\rightarrow Y$ and the map $B^*\rightarrow B^*_1$ is induced by the abelian variety part $i_{\mr{ab}}$ of $i_{\mr{c}}$. The injectivity of $i$ implies that $i_{\mr{c}}$ and $i_{-1}$ are both injective by Lemma \ref{lem2.6}. It follows that $i_{\mr{t}}^*$ is surjective. The injectivity of $i_{\mr{c}}$ implies that $i_{\mr{ab}}:B_1\rightarrow B$ has finite kernel. Hence $i_{\mr{ab}}^*:B^*\rightarrow B_1^*$ is surjective, and so is $i_{\mr{c}}^*$. Then the surjectivity of $f_{\mr{c}}$ follows.

Let $j:A_2\hookrightarrow A$ be the kernel of $f$ in the category of log abelian varieties guaranteed by Proposition \ref{prop2.5}. The proof of Proposition \ref{prop2.5} tells us that the log abelian subvariety $A_2$ could come from a log 1-motive $M_2=[Y_2\rightarrow G_{2\mr{log}}]$ and a homomorphism $(j_{-1},j_0):M_2\rightarrow M$, such that both $j_{-1}$ and $j_0$ are injective, and the semi-abelian variety $G_2$ underlying $ G_{2\mr{log}}$ is the reduced neutral component of $\mr{Ker}(f_{\mr{c}})$. 


Let $Y':=Y_1\times Y_2, G':=G_1\times G_2$, then we have a natural log 1-motive $M'=[Y'\rightarrow G'_{\mr{log}}]$ and a homomorphism $(\alpha_{-1},\alpha_0):M'\rightarrow M$, where $\alpha_{-1}$ is the map $Y_1\times Y_2\rightarrow Y, (y_1,y_2)\mapsto y_1+y_2$, and $\alpha_0$ is the map induced by $\alpha_{\mr{c}}:G_1\times G_2\rightarrow G, (g_1,g_2)\mapsto g_1+g_2$. We claim that $\alpha_{-1}$ is injective and of finite cokernel. Note that $(\gamma_{-1},\gamma_0):=(i^*_{-1}\circ \lambda_{-1}\circ i_{-1},i^*_0\circ\lambda_0\circ i_0)$ is a polarisation on $M_1$ by Lemma \ref{lem2.4}. For $y\in Y_{1\bar{s}}\cap Y_{2\bar{s}}$, $\gamma_{-1}(y)=f_{-1}(y)=0$, hence $0=<0,y>_{1\bar{s}}=<\gamma_{-1}(y),y>_{1\bar{s}}$ implies $y=0$, and $\mr{Ker}(\alpha_{-1})=Y_1\cap Y_2=0$. Hence $\alpha_{-1}$ has finite cokernel due to rank reason.  We also claim that the map $\alpha_{\mr{c}}$ is an isogeny. Let $\gamma_{\mr{c}}:G_1\rightarrow G_1^*$ be the homomorphism corresponding to $\gamma_0$, $\gamma_{\mr{ab}}$ the abelian part of $\gamma_{c}$, and $\gamma_{\mr{t}}$ the torus part of $\gamma_{\mr{c}}$. Since $(\gamma_{-1},\gamma_0)$ is a polarisation, $\gamma_{\mr{ab}}$ is a polarisation and $\gamma_{-1}$ is injective and of finite cokernel. It follows that both $\gamma_{\mr{ab}}$ and $\gamma_{\mr{t}}$ are isogenies. Whence $\gamma_{\mr{c}}=i^*_{\mr{c}}\circ\lambda_{\mr{c}}\circ i_{\mr{c}}$ is also an isogeny. Hence for any $g\in G$, there exists $g_1\in G_1$ such that $i^*_{\mr{c}}\circ\lambda_{\mr{c}}\circ i_{\mr{c}}(g_1)=i^*_{\mr{c}}\circ\lambda_{\mr{c}}(g)$, it follows then $g-i_{\mr{c}}(g_1)\in\mr{Ker}(f_{\mr{c}})$. Since $\mr{Ker}(f_{\mr{c}})/G_2$ is a finite group scheme, the map $\alpha_{\mr{c}}$ is an isogeny. Note that $(\alpha_{-1},\alpha_0)$ induces a homomorphism $\alpha:A_1\times A_2\rightarrow A$ of log abelian varieties. By the equivalence $(1)\Leftrightarrow (4)$ of Proposition \ref{prop2.3}, we deduce that $\alpha$ is an isogeny.
\end{proof}

\begin{rmk}\label{rmk2.3}
Since abelian varieties over a field are always projective, they carry an ample line bundle, hence they are always polarisable. For log abelian varieties over a log point, they admit a polarisation after base change to the algebraic closure of the base field by definition. However it is not clear to the author if they actually carry a polarisation over the base log point. But he does think, over a log point, log abelian variety admitting a polarisation serves as the right counterpart of abelian variety for at least two reasons. Firstly, the canonical 1-parameter log abelian variety degeneration (\cite{zha1}) of an abelian variety transports a polarisation of the generic fibre to the special fibre. In other words the special fibre (which is a log abelian variety over a log point) as the degeneration of the generic fibre (which is an abelian variety over a trivial log point) is necessarily polarisable. Secondly, a polarisation is needed in the proof of Poincar\'e complete reducibility theorem (see Theorem \ref{thm2.2}), and we know that Poincar\'e complete reducibility theorem for abelian varieties plays a very important role in the theory of abelian varieties.
\end{rmk}

\begin{defn}\label{defn2.3}
Let $A$ be a log abelian variety over $S$, a log abelian subvariety of $A$ is a subsheaf of $A$ which is also a log abelian variety. The log abelian variety $A$ is simple if it has no non-zero proper log abelian subvariety. In other words, if $A_1$ is a log abelian variety properly contained in $A$, then $A_1$ is zero. 
\end{defn}

\begin{lem}\label{lem2.7}
Let $f:A\rightarrow A'$ be a non-zero homomorphism between log abelian varieties.
\begin{enumerate}[(1)]
\item If both $A$ and $A'$ are simple, then $f$ is an isogeny.
\item If $f$ is an isogeny, then there exists an isogeny $g:A'\rightarrow A$ and a positive integer $n$ such that $g\circ f=n_{A}$.
\end{enumerate}
\end{lem}
\begin{proof}
We prove part (1) first. Let $A_1$ be the kernel of $f$ in the category of log abelian varieties, see Proposition \ref{prop2.5} for the construction of $A_1$. Since $A$ is simple, $A_1$ has to be zero. Let $F$ be the kernel of $f$ in the category of sheaves of abelian groups, then $F$ lies in $(\mr{fin}/S)_r$ by part (2) of Lemma \ref{lem2.4}. Then $A/F$ gives rise to a log abelian subvariety of $A'$ by part (3) of Proposition \ref{prop2.2}. We must have $A'=A/F$ by the simplicity of $A'$. It follows that $f$ is an isogeny. This shows part (1).

Now we show part (2). Let $n$ be a positive integer which kills $F$. Then there exists an epimorphism $g:A'\rightarrow A$ such that $g\circ f=n_A$. The kernel-cokernel exact sequence gives a short exact sequence $0\rightarrow F\rightarrow A[n]\rightarrow \mr{Ker}(g)\rightarrow 0$. We get $\mr{Ker}(g)\in (\mr{fin}/S)_r$ by part (2) of Proposition \ref{prop2.1}. Hence $g$ is an isogeny.
\end{proof}

\begin{cor}\label{cor2.1}
Let $A$ be a log abelian variety over $S$ admitting a polarisation. Then $A$ is isogenous to a product $A_1^{n_1}\times\cdots \times A_r^{n_r}$, where the $A_i$'s are simple log abelian varieties and not isogenous to each other. The isogeny type of the $A_i$ and the integers $n_i$'s are uniquely determined.
\end{cor}
\begin{proof}
If $A$ is simple, there is nothing to prove. Otherwise, by Theorem \ref{thm2.2}, there exists nonzero log abelian subvarieties  $A'$ and $A''$ of $A$ such that $A$ is isogeneous to $A'\times A''$. By Lemma \ref{lem2.4} and Lemma \ref{lem2.6}, both $A'$ and $A''$ admit polarisations. Hence applying Theorem \ref{thm2.2} repeatedly, we have that $A$ is isogenous to a product $A_1^{n_1}\times\cdots \times A_r^{n_r}$ for some simple log abelian varieties $A_i$ and some positive integers $n_i$. By part (1) of Lemma \ref{lem2.7}, such a decomposition is unique up to isogeny.
\end{proof}

\begin{defn}\label{defn2.4}
Let $A,A'$ be two log abelian varieties over $S$, we abbreviate $\mr{Hom}_{S^{\mr{log}}_{\mr{fl}}}(A,A')=\mr{Hom}_{S^{\mr{cl}}_{\mr{\acute{E}t}}}(A,A')$ as $\mr{Hom}(A,A')$. We define $\mr{Hom}^0(A,A')$ as $\mr{Hom}(A,A')\otimes_{\Z}\Q$, and $\mr{End}^0(A)$ as $\mr{End}(A)\otimes_{\Z}\Q=\mr{Hom}(A,A)\otimes_{\Z}\Q$. We define the category $\mr{LAV}^0_S$ of log abelian varieties up to isogeny over $S$, by localising the category $\mr{LAV}_S$ of log abelian varieties over $S$ at the class of isogenies.
\end{defn}

\begin{cor}\label{cor2.2}
Let $A$ be a log abelian variety over $S$ admitting a polarisation. If $A$ is simple, the ring $\mr{End}^0(A)$ is a division ring. In general, if $A$ is isogenous to $A_1^{n_1}\times\cdots\times A_r^{n_r}$ with $A_i$ simple and not isogenous to each other, and $D_i=\mr{End}^0(A_i)$, then $\mr{End}^0(A)=M_{n_1}(D_1)\times\cdots\times M_{n_r}(D_{r})$.
\end{cor}
\begin{proof}
For $A$ simple, let $f$ be a nonzero endomorphism of $A$. Then $f$ is an isogeny by part (1) of Lemma \ref{lem2.7}. By part (2) of Lemma \ref{lem2.7}, $f$ is invertible in the ring $\mr{End}^0(A)$. Hence the ring $\mr{End}^0(A)$ is a division ring. By part (1) of Lemma \ref{lem2.7}, we have $\mr{Hom}^0(A,A')=0$ for two non-isogenous simple log abelian varieties. Hence the second part follows.
\end{proof}

\begin{lem}\label{lem2.8}
The abelian group $\mr{Hom}(A,A')$ is torsion-free.
\end{lem}
\begin{proof}
Let $f\in \mr{Hom}(A,A')$ such that $nf=0$ for some positive integer $n$. Since $0=nf=f\circ n_A$ and $n_A$ is surjective, $f$ must be zero. Hence $\mr{Hom}(A,A')$ is torsion-free.
\end{proof}

\begin{defn}\label{defn2.5}
\begin{enumerate}[(1)]
\item Let $f:A\rightarrow A'$ be an isogeny between two log abelian varieties over $S$. The degree $\mr{deg}(f)$ of $f$ is defined to be the rank of the finite log group object $\mr{Ker}(f)$. By convention, if $f$ is not an isogeny, we let $\mr{deg}(f)=0$.
\item Let $f_{-1}:Y\rightarrow Y'$ be a monomorphism with finite cokernel between two \'etale locally finite rank free constant sheaf, the degree $\mr{deg}(f_{-1})$ of $f_{-1}$ is defined to be the determinant of $f_{-1}$. By convention, if $f$ is not injective of finite cokernel, we let $\mr{deg}(f)=0$.
\item Let $f_{\mr{c}}:G\rightarrow G'$ be an isogeny between semi-abelian varieties, the degree $\mr{deg}(f_{\mr{c}})$ of $f_{\mr{c}}$ is defined to be the rank of the finite group scheme $\mr{Ker}(f_{\mr{c}})$. By convention, if $f_{\mr{c}}$ is not an isogeny, we let $\mr{deg}(f_{\mr{c}})=0$.
\end{enumerate}
\end{defn}

\begin{lem}\label{lem2.9}
Let $f:A\rightarrow A$ be a homomorphism between two log abelian varieties over $S$. Let $f_{-1}:Y\rightarrow Y$ and $f_{\mr{c}}:G\rightarrow G$ be the homomorphisms induced by $f$ as in Proposition \ref{prop1.6}, and let $f_{\mr{t}}:T\rightarrow T$ and $f_{\mr{ab}}:B\rightarrow B$ be the homomorphisms induced by $f_{\mr{c}}$ on torus parts and abelian variety parts respectively. Then:
\begin{enumerate}[(1)]
\item $\mr{deg}(f)=\mr{deg}(f_{-1})\mr{deg}(f_{\mr{t}})\mr{deg}(f_{\mr{ab}})$;
\item let $g:A\rightarrow A$ be another homomorphism, and let $h=f+g$, then $h_{-1}=f_{-1}+g_{-1}$, $h_{\mr{t}}=f_{\mr{t}}+g_{\mr{t}}$ and $h_{\mr{ab}}=f_{\mr{ab}}+g_{\mr{ab}}$.
\end{enumerate} 
\end{lem}
\begin{proof}
Part (2) is obvious. We only need to show part (1). 

If $f$ is not an isogeny, then $\mr{deg}(f_{-1})\mr{deg}(f_{\mr{t}})\mr{deg}(f_{\mr{ab}})=0=\mr{deg}(f)$ by Proposition \ref{prop2.3}. Now we suppose that $f$ is an isogeny, so are $f_{\mr{c}}$, $f_{\mr{t}}$ and $f_{\mr{ab}}$.  Also we have $f_{-1}$ is injective and of finite cokernel. By diagram (\ref{eq1.10}), we get a short exact sequence $$0\rightarrow \mr{Ker}(f_{\mr{c}})\rightarrow \mr{Ker}(f)\rightarrow \mr{Ker}(f_{\mr{d}})\rightarrow 0.$$ Similarly, we have another short exact sequence $$0\rightarrow \mr{Ker}(f_{\mr{t}})\rightarrow \mr{Ker}(f_{\mr{c}})\rightarrow \mr{Ker}(f_{\mr{ab}})\rightarrow 0.$$
By diagram (\ref{eq1.12}), we get $\mr{Ker}(f_{\mr{d}})\cong \mr{Coker}(f_{-1})$. Then $$\mr{deg}(f)=\mr{deg}(f_{-1})\mr{deg}(f_{\mr{t}})\mr{deg}(f_{\mr{ab}}).$$ 
\end{proof}

\begin{thm}\label{thm2.3}
The function $f\mapsto \mr{deg}(f)$ on $\mr{End}(A)$ extends to a homogeneous polynomial function of degree $2g$ on $\mr{End}^0(A)$, where $g$ is the dimension of $A$.
\end{thm}
\begin{proof}
Since for any $f\in \mr{End}(A)$ and $n\in\Z$, $$\mr{deg}(nf)=\mr{deg}(n_A)\cdot \mr{deg}(f)=n^{2g}\cdot \mr{deg}(f),$$
it suffices to show that for $f,g\in \mr{End}(A)$, the function $P(n)=\mr{deg}(nf+g)$ is a polynomial function. By Lemma \ref{lem2.9}, we are reduced to show the functions $\mr{deg}(nf_{\mr{t}}+g_{\mr{t}})$, $\mr{deg}(nf_{\mr{ab}}+g_{\mr{ab}})$ and $\mr{deg}(nf_{-1}+g_{-1})$ are all polynomial functions. The case for $\mr{deg}(nf_{\mr{ab}}+g_{\mr{ab}})$ is a standard result for abelian varieties, see \cite[\S 19, Thm. 2]{mum2}. And $\mr{deg}(nf_{-1}+g_{-1})$ as a determinant function is clearly a polynomial function. The case for $\mr{deg}(nf_{\mr{t}}+g_{\mr{t}})$ is reduced to the case for $\mr{deg}(nf_{-1}+g_{-1})$ by taking the character groups of the tori.
\end{proof}

\begin{defn}\label{defn2.6}
Let $l$ be a prime number which is coprime to the characteristic of $k$. The $l$-adic Tate module of $A$ is defined to be 
$$T_l(A)_{\bar{s}(\mr{k\acute{e}t})}=\varprojlim_n A[l^n]_{\bar{s}(\mr{k\acute{e}t})},$$
where $\bar{s}(\mr{k\acute{e}t})$ denotes a log geometric point of $S$ for the log \'etale topology. Here we use the notation $\bar{s}(\mr{k\acute{e}t})$ for the sake of coherence with \cite[18.9]{k-k-n4}, and the log \'etale topology is called the Kummer \'etale topology there.
\end{defn}

Let $\pi_1^{\mr{log}}$ be the log fundamental group of $S$. By Proposition \ref{prop2.4}, we have $T_l(A)_{\bar{s}(\mr{k\acute{e}t})}$ is a free $\Z_l$-module of rank $2g$ endowed with a continuous $\pi_1^{\mr{log}}$-action. Any homomorphism $f:A\rightarrow A'$ induces a homomorphism $$T_l(f):T_l(A)_{\bar{s}(\mr{k\acute{e}t})}\rightarrow T_l(A')_{\bar{s}(\mr{k\acute{e}t})}$$
which is $\pi_1^{\mr{log}}$-equivariant. It follows that we have a functor
$$T_l:\mr{LAV}_S\longrightarrow (\pi_1^{\mr{log}},\Z_l)\mr{-Mod}$$
from the category of log abelian varieties over the log point $S$ to the category of finite rank $\Z_l$-modules with continuous $\pi_1^{\mr{log}}$-action. In particular, the functor $T_l$ gives rise to a homomorphism 
$$\mr{Hom}(A,A')\rightarrow \mr{Hom}_{(\pi_1^{\mr{log}},\Z_l)\mr{-Mod}}(T_l(A)_{\bar{s}(\mr{k\acute{e}t})},T_l(A')_{\bar{s}(\mr{k\acute{e}t})}).$$ 
The latter is clearly a $\Z_l$-submodule of $\mr{Hom}_{\Z_l}(T_l(A)_{\bar{s}(\mr{k\acute{e}t})},T_l(A')_{\bar{s}(\mr{k\acute{e}t})})$ which is of finite $\Z_l$-rank. Moreover, we have the following canonical homomorphism
$$T_l:\mr{Hom}(A,A')\otimes_{\Z}\Z_l\rightarrow \mr{Hom}_{\Z_l}(T_l(A)_{\bar{s}(\mr{k\acute{e}t})},T_l(A')_{\bar{s}(\mr{k\acute{e}t})}).$$
We are going to use this map to investigate the finiteness of $\mr{Hom}(A,A')$.

\begin{thm}\label{thm2.4}
For $A,A'$ two log abelian varieties over $S$ admitting a polarisation, $\mr{Hom}(A,A')$ is a finitely generated free abelian group, and the canonical map $$T_l:\mr{Hom}(A,A')\otimes_{\Z}\Z_l\rightarrow \mr{Hom}_{\Z_l}(T_l(A)_{\bar{s}(\mr{k\acute{e}t})},T_l(A')_{\bar{s}(\mr{k\acute{e}t})})$$ is injective, where $l$ is a prime number different from the characteristic of $k$.
\end{thm}
\begin{proof}
We have already proven the degree function on $\mr{End}^0(A)$ is a homogeneous polynomial function of degree $2g$ in Theorem \ref{thm2.3}. Now the proof of \cite[\S 19, Thm. 3]{mum2} works verbatim here.
\end{proof}

\begin{cor}\label{cor2.3}
Let $A,A'$ be two log abelian varieties over $S$. Then the canonical map $$T_l:\mr{Hom}(A,A')\otimes_{\Z}\Z_l\rightarrow \mr{Hom}_{\Z_l}(T_l(A)_{\bar{s}(\mr{k\acute{e}t})},T_l(A')_{\bar{s}(\mr{k\acute{e}t})})$$ is injective.
\end{cor}
\begin{proof}
This follows from Theorem \ref{thm2.4}.
\end{proof}

\begin{cor}\label{cor2.4}
Let $A,A'$ be two log abelian varieties over $S$. Then we have $\mr{Hom}(A,A')\cong\Z^r$ with $r\leq 4\mr{dim}A\cdot\mr{dim}A'$.
\end{cor}

\begin{cor}\label{cor2.5}
Let $A$ be a log abelian variety over $S$ admitting a polarisation. Then $\mr{End}^0(A)$ is a finite-dimensional semisimple algebra over $\Q$.
\end{cor}
\begin{proof}
This follows from Corollary \ref{cor2.2} and Corollary \ref{cor2.4}.
\end{proof}

\appendix
\section{Appendix: Log finite flat group schemes}
Since the theory of log finite flat group schemes is not well-known, we collect some results about them in this appendix. These results are all due to Kato, and the main references are \cite{kat4} and \cite{mad1}. 

Let $S$ be an fs log scheme. We recall several kinds of finite group objects on $S^{\mr{log}}_{\mr{fl}}$ defined by Kato.

\begin{defn}\label{defn-app.1}
The category $(\mr{fin}/S)_c$ is the full subcategory of the category of sheaves of finite abelian groups over $S^{\mr{log}}_{\mr{fl}}$ consisting of objects which are representable by a classical finite flat group scheme over $S$. Here classical means the log structure of the representing log scheme is the one induced from $S$. 

The category $(\mr{fin}/S)_f$ is the full subcategory of the category of sheaves of finite abelian groups over $S^{\mr{log}}_{\mr{fl}}$ consisting of objects which are representable by a classical finite flat group scheme over a log flat cover of $S$. For $F\in (\mr{fin}/S)_f$, let $U\rightarrow S$ be a log flat cover of $S$ such that $F_U:=F\times_S U\in (\mr{fin}/S)_c$, then the rank of $F$ is defined to be  the rank of $F_U$ over $U$.

The category $(\mr{fin}/S)_r$ is the full subcategory of $(\mr{fin}/S)_f$ consisting of objects which are representable by a log scheme over $S$.

Let $F\in (\mr{fin}/S)_f$, the Cartier dual of $F$ is the sheaf $F^*:=\mc{H}om_{S^{\mr{log}}_{\mr{fl}}}(F,\Gm)$. By the definition of $(\mr{fin}/S)_f$, it is clear that $F^*\in (\mr{fin}/S)_f$.

The category $(\mr{fin}/S)_d$ is the full subcategory of $(\mr{fin}/S)_r$ consisting of objects whose Cartier duals also lie in $(\mr{fin}/S)_r$.
\end{defn}

\subsection{Logarithmic fundamental group}
We have the following well-known theorem.

\begin{thm}\label{thm-app.1}
Let $\mathring{S}$ be a locally noetherian connected scheme. Let $\mr{fin}^{\mr{et}}_{\mathring{S}}$ denote the category of finite \'etale group schemes over $\mathring{S}$, $\mr{LC}(\mathring{S}_{\mr{\acute{E}t}})$ (resp. $\mr{LC}(\mathring{S}_{\mr{fl}})$) the category of locally constant sheaves of finite abelian groups for the \'etale (resp. flat) topology, and $\pi_1(\mathring{S})\mr{-fMod}$ the category of finite abelian groups endowed with a continuous $\pi_1(\mathring{S})$-action.
\begin{enumerate}[(1)]
\item By the theory of fundamental group, there are equivalences of categories:
$$\mr{fin}^{\mr{et}}_{\mathring{S}}\xrightarrow{\cong} \mr{LC}(\mathring{S}_{\mr{\acute{E}t}})\xrightarrow{\cong} \pi_1(\mathring{S})\mr{-fMod}.$$
\item By flat descent, we further have an equivalence
$$\mr{LC}(\mathring{S}_{\mr{\acute{E}t}})\xrightarrow{\cong}\mr{LC}(\mathring{S}_{\mr{fl}}).$$ 
\end{enumerate}
\end{thm}

For the theory of logarithmic fundamental group, we have the following analogue of Theorem \ref{thm-app.1} which is due to Kato.

\begin{thm}\label{thm-app.2}
Let $S$ be an fs log scheme with underlying scheme locally noetherian and connected. Let $\pi_1^{\mr{log}}(S)$ be the logarithmic fundamental group of $S$, see \cite[4.6]{ill1} for its definition. Let $\mr{fin}^{\mr{ket}}_{S}$ be the subcategory of $(\mr{fin}/S)_r$ consisting of objects which are Kummer log \'etale over $S$, $\mr{LC}(S^{\mr{log}}_{\mr{\acute{E}t}})$ (resp. $\mr{LC}(S^{\mr{log}}_{\mr{fl}})$) the category of locally constant sheaves of finite abelian groups on $S^{\mr{log}}_{\mr{\acute{E}t}}$ (resp. $S^{\mr{log}}_{\mr{fl}}$), and $\pi_1^{\mr{log}}(S)\mr{-fMod}$ the category of finite abelian groups endowed with a continuous $\pi_1^{\mr{log}}(S)$-action.
Then we have the following equivalences of categories:
\begin{enumerate}[(1)]
\item $$\mr{fin}^{\mr{ket}}_{S}\xrightarrow{\cong} \mr{LC}(S^{\mr{log}}_{\mr{\acute{E}t}})\xrightarrow{\cong} \pi_1^{\mr{log}}(S)\mr{-fMod};$$
\item $$\mr{LC}(S^{\mr{log}}_{\mr{\acute{E}t}})\xrightarrow{\cong}\mr{LC}(S^{\mr{log}}_{\mr{fl}}).$$ 
\end{enumerate}
\end{thm}
\begin{proof}
For part (1), see \cite[\S4]{ill1}. For part (2), see \cite[Thm. 1.4.5. (2)]{mad1}.
\end{proof}

\subsection{Structure of log finite flat group schemes}
\begin{prop}\label{prop-app.1}
Suppose that the underlying scheme of $S$ is locally noetherian. Then the category $(\mr{fin}/S)_f$ (resp. $(\mr{fin}/S)_r$, resp. $(\mr{fin}/S)_d$) is closed under extensions in the category of sheaves of abelian groups on $S^{\mr{log}}_{\mr{fl}}$.
\end{prop}
\begin{proof}
See \cite[Prop. 2.3]{kat4}. 
\end{proof}

\begin{lem}\label{lem-app.1}
Assume that the underlying scheme of $S$ is the spectrum of a henselian local ring. Let $F\in (\mr{fin}/S)_f$, then there is a unique short exact sequence 
\begin{equation}\label{eq-app.1}
0\rightarrow F^{\circ}\rightarrow F\rightarrow F^{\mr{et}}\rightarrow 0
\end{equation} 
in $(\mr{fin}/S)_f$, such that over any log flat cover $S'\rightarrow S$ with $F_{S'}\in (\mr{fin}/S')_c$, this sequence restricts to the classical connected-\'etale sequence.
\end{lem}
\begin{proof}
See \cite[2.6]{kat4}, or \cite[Lem. 2.1.6]{mad1}.
\end{proof}

\begin{lem}\label{lem-app.2}
Let the assumption be as in Lemma \ref{lem-app.1}. Assume further that $F$ lies in $(\mr{fin}/S)_r$ and its underlying scheme is connected. Then $F$ actually lies in $(\mr{fin}/S)_c$.
\end{lem}
\begin{proof}
See \cite[Prop. 2.1.7]{mad1}, see also \cite[Lem. 2.8]{kat4} for the noetherian strict henselian case.
\end{proof}

\begin{prop}\label{prop-app.2}
Let the notation and the assumption be as in Lemma \ref{lem-app.1}. We further assume that the underlying scheme is noetherian.
\begin{enumerate}[(1)]
\item The sheaf $F^{\mr{et}}$ lies in $\mr{fin}^{\mr{ket}}_{S}$, in particular it lies in $(\mr{fin}/S)_r$.
\item The sheaf $F$ lies in $(\mr{fin}/S)_r$ if and only if $F^{\circ}$ lies in $(\mr{fin}/S)_c$.
\end{enumerate}
\end{prop}
\begin{proof}
Part (1) follows from Theorem \ref{thm-app.2}.

Suppose that $F^{\circ}\in(\mr{fin}/S)_c$, then $F$ as an $F^{\circ}$-torsor over $F^{\mr{et}}$ is representable by \cite[Thm. 9.1]{kat2}. Hence $F\in (\mr{fin}/S)_r$. This proves one direction of part (2).

Conversely suppose that $F\in (\mr{fin}/S)_r$, we want to show that $F^{\circ}\in (\mr{fin}/S)_c$. As the kernel of $F\rightarrow F^{\mr{et}}$, $F^{\circ}$ is representable, hence lies in $(\mr{fin}/S)_r$. We claim that the underlying scheme of $F^{\circ}$ is connected. Taking a covering $U\rightarrow S$ in $S^{\mr{log}}_{\mr{fl}}$ such that the underlying scheme of $U$ is finite over that of $S$, and $F^{\circ}_U\in (\mr{fin}/U)_c$. Since the underlying scheme of $F^{\circ}_U$ is connected, so is that of $F^{\circ}$. Now Lemma \ref{lem-app.2} applies, and we get $F^{\circ}\in (\mr{fin}/S)_c$.
\end{proof}

\begin{rmk}\label{rem-app.1}
The proof of Proposition \ref{prop-app.2} is the same as the proof of the strict henselian case in \cite[Prop. 2.7]{kat4}. However there is a small condition missing in the proof in \cite[Prop. 2.7]{kat4}, so we give a complete proof here.
\end{rmk}

Let $S$ be an fs log scheme with underlying scheme $\mr{Spec}R$, where $R$ is a noetherian henselian local ring with residue characteristic $p>0$. Let $x$ be the closed point of $S$, and suppose that $S$ admits a global chart $P\rightarrow\mc{O}_S$ which induces an isomorphism $P\rightarrow (M_S/\mc{O}_S^{\times})_{\bar{x}}$. Here $P$ is an fs monoid, and $\bar{x}$ is a classical geometric point over $x$. By Proposition \ref{prop-app.2}, any $F\in (\mr{fin}/S)_r$ is an extension of $F^{\mr{et}}\in \mr{fin}^{\mr{ket}}_{S}$ by $F^{\circ}\in (\mr{fin}/S)_c$. By Theorem \ref{thm-app.2}, $F^{\mr{et}}$ can be understood by the theory of logarithmic fundamental group. Let $S'$ be a Galois cover of $S$ with Galois group $G:=\pi_1^{\mr{log}}(S)/\pi_1^{\mr{log}}(S')$, such that $F^{\mr{et}}\times_S S'\in (\mr{fin}/S')_c$. Then one can recover $F$ from $F\times_S S'$ and the $G$-action on $F\times_S S'$ induced from the $G$-action on $F^{\mr{et}}\times_S S'$.

In order to understand $F$, we may assume $F^{\mr{et}}\in (\mr{fin}/S)_c$ after replacing $S$ by $S'$. Hence to understand the category $(\mr{fin}/S)_r$, we are reduced to understand the extensions of a classical finite \'etale group scheme by a classical connected finite flat group scheme on $S^{\mr{log}}_{\mr{fl}}$. This is done by Kato's classification theorem of logarithmic finite flat group schemes as follows.

\begin{thm}\label{thm-app.3}
Let $S$ be as above. Let $G$ be a classical finite \'etale group scheme over $S$, $H$ a classical finite flat group scheme over $S$. We denote by $\mathfrak{Ext}_{S^{\mr{log}}_{\mr{fl}}}(G,H)$ (resp. $\mathfrak{Ext}_{S^{\mr{cl}}_{\mr{fl}}}(G,H)$) the category of extensions of $G$ by $H$ on $S^{\mr{log}}_{\mr{fl}}$ (resp. $S^{\mr{cl}}_{\mr{fl}}$), and $\mathfrak{Hom}_{S^{\mr{cl}}_{\mr{fl}}}(G(1),H)\otimes P^{\mr{gp}}$ the discrete category associated to the group $\mr{Hom}_{S^{\mr{cl}}_{\mr{fl}}}(G(1),H)\otimes P^{\mr{gp}}$. Then we have an equivalence of categories
$$\Phi:\mathfrak{Ext}_{S^{\mr{cl}}_{\mr{fl}}}(G,H)\times \mathfrak{Hom}_{S^{\mr{cl}}_{\mr{fl}}}(G(1),H)\otimes P^{\mr{gp}} \xrightarrow{\cong} \mathfrak{Ext}_{S^{\mr{log}}_{\mr{fl}}}(G,H) .$$
\end{thm}
\begin{proof}
See \cite[Thm. 2.3.1]{mad1}. See also \cite[Thm. 3.3]{kat4} for the strict henselian case.
\end{proof}

\bibliographystyle{alpha}
\bibliography{mybib}

\Addresses
\end{document}